\documentclass[11pt]{extarticle}

\oddsidemargin \evensidemargin
\usepackage[utf8]{inputenc} 
\usepackage[T1]{fontenc}    
\usepackage{hyperref}       
\usepackage{url}            
\usepackage{booktabs}       
\usepackage{amsfonts}       
\usepackage{nicefrac}       
\usepackage{microtype}      
\usepackage{lipsum}
\usepackage{array,multirow}
\usepackage{blindtext}
\usepackage{graphicx}
\usepackage{xcolor}
\usepackage{amsmath}
\usepackage{mathtools}
\usepackage{comment}
\usepackage{amssymb,amsfonts}
\usepackage{amsthm}
\usepackage[tight,footnotesize]{subfigure}
\usepackage{caption}
\usepackage{booktabs}
\usepackage{comment}

\usepackage{tikz}
\usepackage{tikz-cd}
\usetikzlibrary{shapes,backgrounds}
\usepackage[utf8]{inputenc}
\usepackage{amsmath,xcolor,graphicx,array,multicol,manfnt}
\usepackage{tkz-euclide}
\usepackage{float}

\usepackage{adjustbox}

\renewcommand{\em}{\it}

\newcommand{\VV}{\mathbf{V}}
\newcommand{\II}{\mathbf{I}}
\newcommand{\ee}{\mathbf{e}}
\newcommand{\xx}{\mathbf{x}}
\newcommand{\yy}{\mathbf{y}}
\newcommand{\zz}{\mathbf{z}}

\newcommand{\R}{\mathbf{R}}
\newcommand{\E}{\mathbf{E}}
\newcommand{\I}{\mathbf{I}}
\newcommand{\tran}{\mathbf{t}}

\newcommand{\ZZ}{\mathbb{Z}}
\newcommand{\QQ}{\mathbb{Q}}
\newcommand{\RR}{\mathbb{R}}
\newcommand{\CC}{\mathbb{C}}
\newcommand{\PP}{\mathbb{P}}
\newcommand{\GG}{\mathbb{G}}

\newcommand{\bs}{\boldsymbol}

\newcommand{\red}{\textcolor{red}}

\definecolor{darkpastelgreen}{rgb}{0.01, 0.75, 0.24}

\DeclareMathOperator{\cent}{Cent}

\DeclareMathOperator{\sing}{Sing}
\DeclareMathOperator{\sym}{\it Sym}

\DeclareMathOperator{\SO}{\rm SO}
\DeclareMathOperator{\SE}{\rm SE}

\DeclareMathOperator{\tr}{{\rm tr}}

\DeclareMathOperator{\Gal}{Gal}

\DeclareMathOperator{\rank}{rk}

\DeclareMathOperator{\Aut}{Aut}

\newcommand{\galclo}[1]{#1^{\mathrm{gal}}}
\newcommand{\ratnoname}[2]{#1 \dashrightarrow #2}

\newcommand{\mon}{\mathrm{Mon}}

\newcommand{\gal}{\mathrm{Gal}}

\newcommand{\SOCC}{\SO_\CC}

\newcommand{\SECC}{\SE_\CC}
\newcommand{\SERR}{\SE_\RR}

\theoremstyle{definition}
\newtheorem{thm}{Theorem}[section]
\newtheorem{result}{Result}[section]
\newtheorem{defn}[thm]{Definition}
\newtheorem{prop}[thm]{Proposition}

\newtheorem{example}[thm]{Example}
\newtheorem{remark}[thm]{Remark}
\newtheorem{proposition}[thm]{Proposition}

\DeclarePairedDelimiter{\norm}{\lVert}{\rVert}

\newcommand{\rat}[3]{#1 \colon #2 \dashrightarrow #3}

\newcommand\restr[2]{{
  \left.\kern-\nulldelimiterspace 
  #1 
  \vphantom{\big|} 
  \right|_{#2} 
}}

\newcommand{\xdashrightarrow}[2][]{\ext@arrow 0359\rightarrowfill@@{#1}{#2}}
\def\acts{\curvearrowright}

\newcommand{\cam}[2]{[#1\mid #2]}

\newcommand{\wt}[1]{\widetilde{#1}}
\newcommand{\permshort}[4]{\left(\begin{array}{ccc}
#1 & \cdots & #2 \\
#3 & \cdots & #4 
\end{array}\right)}
\renewcommand{\matrix}[1]{\begin{bmatrix} #1 \end{bmatrix}}

\usepackage{import}
\usepackage{xifthen}
\usepackage{pdfpages}
\usepackage{transparent}

\definecolor{blue(pigment)}{rgb}{0.2, 0.2, 0.8}
\newcommand{\deff}[1]{{\em \color{blue(pigment)} #1}}

\newcommand{\Ess}{\mathcal{E}}
\newcommand{\calH}{\mathcal{H}}
\newcommand{\HH}{\mathbf{H}}
\renewcommand{\SS}{\mathbf{S}}
\newcommand{\N}{\mathbf{N}}
\newcommand{\nn}{\mathbf{n}}
\newcommand{\dd}{\mathbf{d}}
\newcommand{\WW}{\mathbf{W}}

\newcommand{\agline}[2]{\overline{#1 \, #2}}

\newcommand{\twoviewhomography}{\vcenter{\hbox{\includegraphics[scale = 0.06]{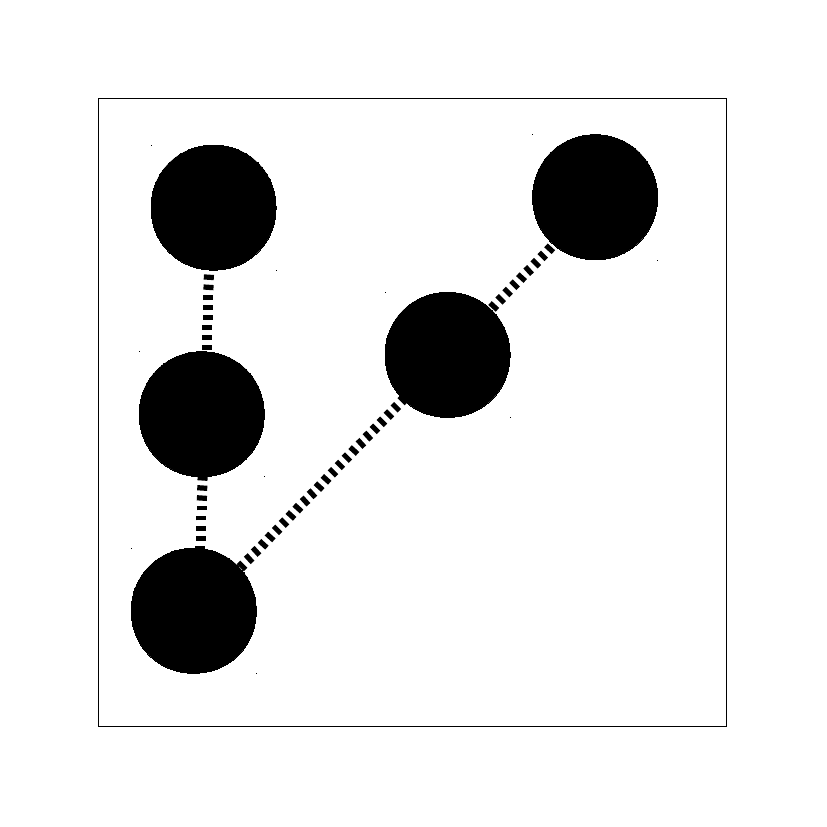}}} }

\newcommand{\threeviewhomography}{\vcenter{\hbox{\includegraphics[scale = 0.06]{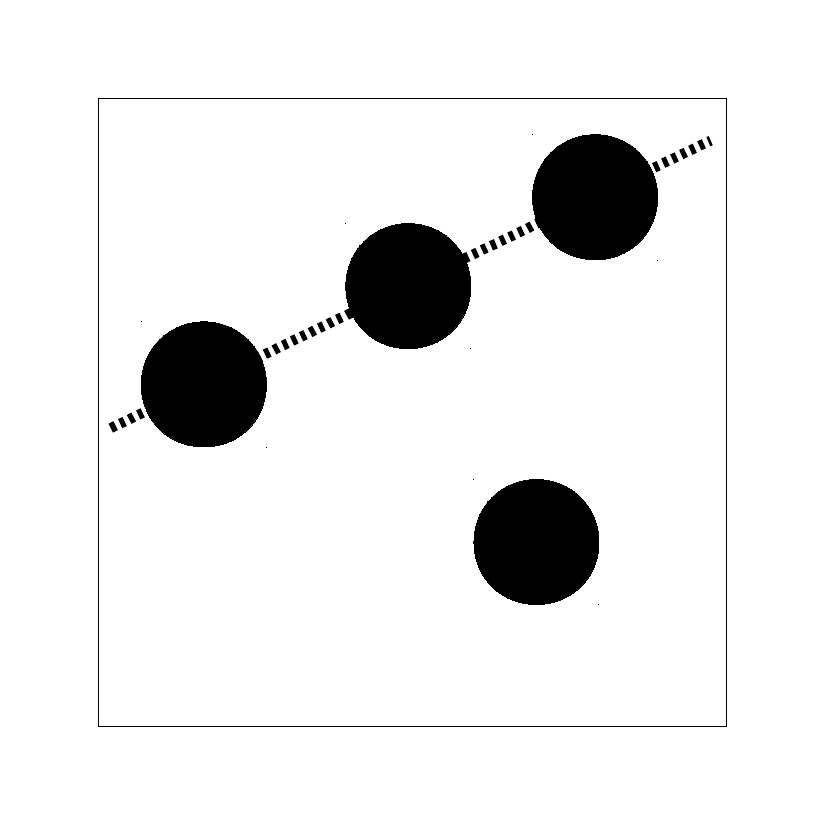}}} }

\newcommand{\fivepointproblem}{\vcenter{\hbox{\includegraphics[scale = 0.06]{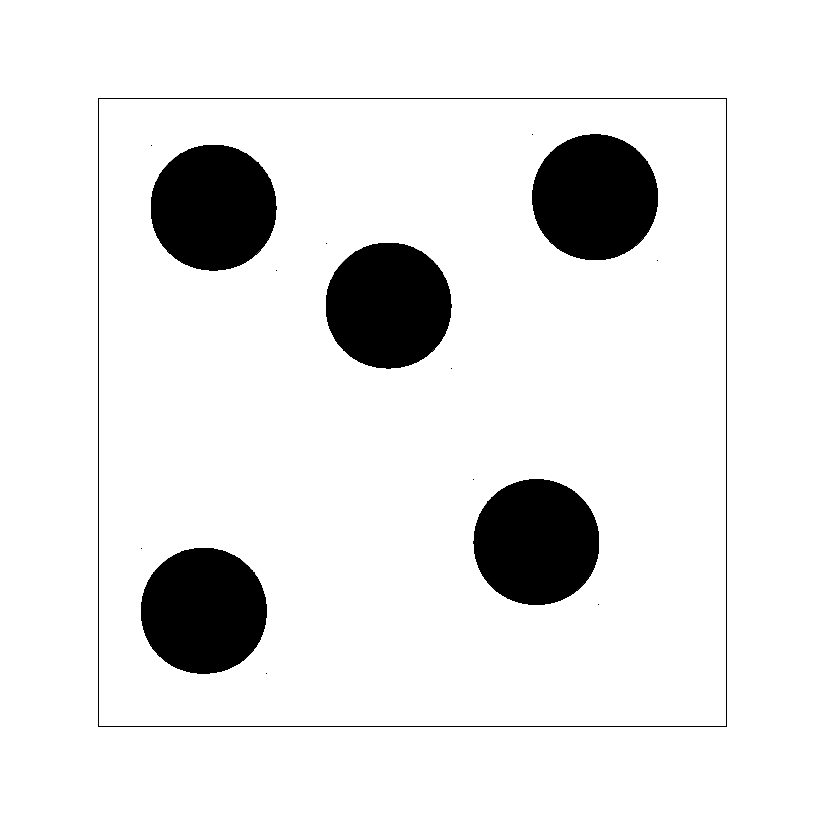}}} }

\newcommand{\chicago}{\vcenter{\hbox{\includegraphics[scale = 0.06]{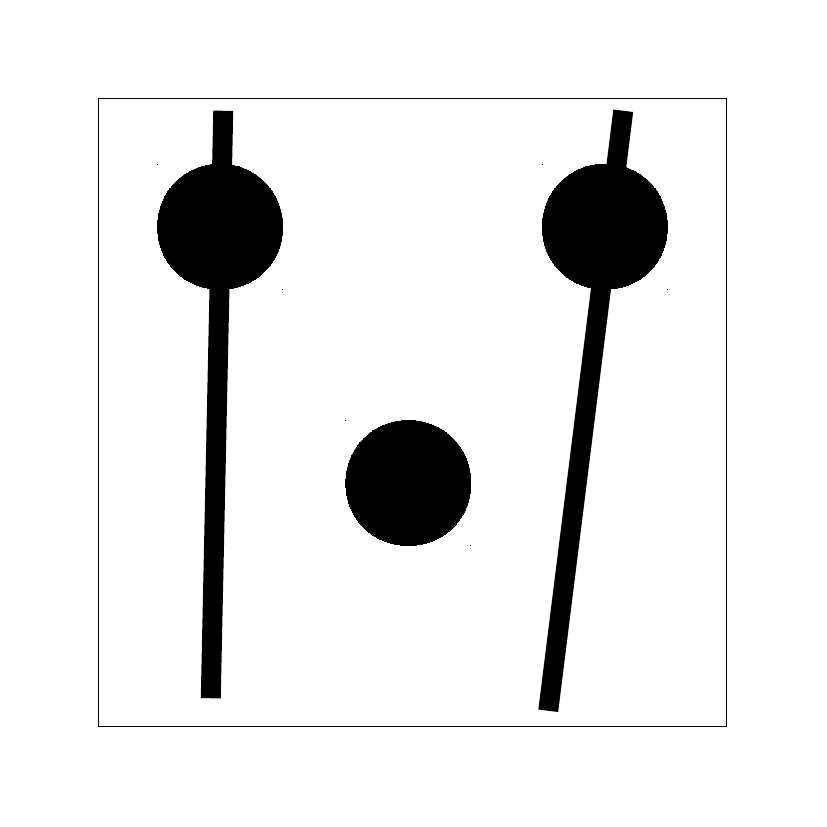}}} }

\newcommand{\cleveland}{\vcenter{\hbox{\includegraphics[scale = 0.06]{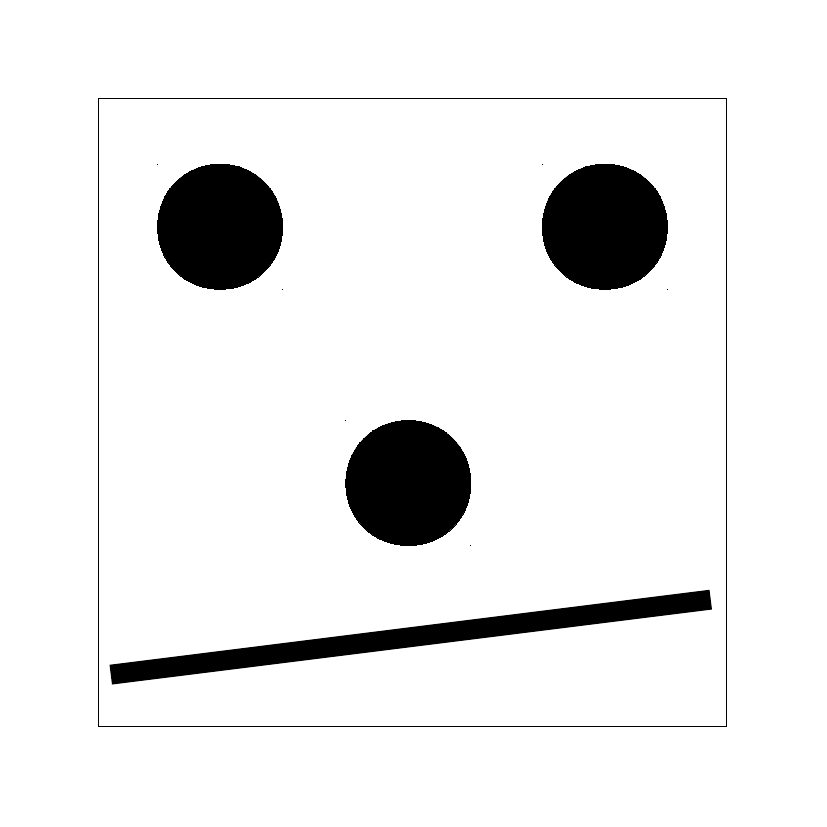}}} }

\newcommand{\partialminimal}{\vcenter{\hbox{\includegraphics[scale = 0.06]{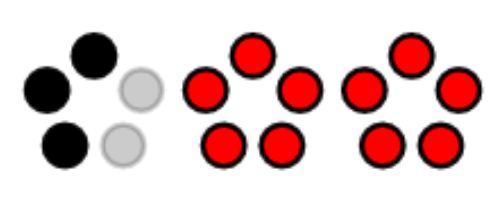}}} }

\newcommand{\fivepointcollinear}{\vcenter{\hbox{\includegraphics[scale = 0.06]{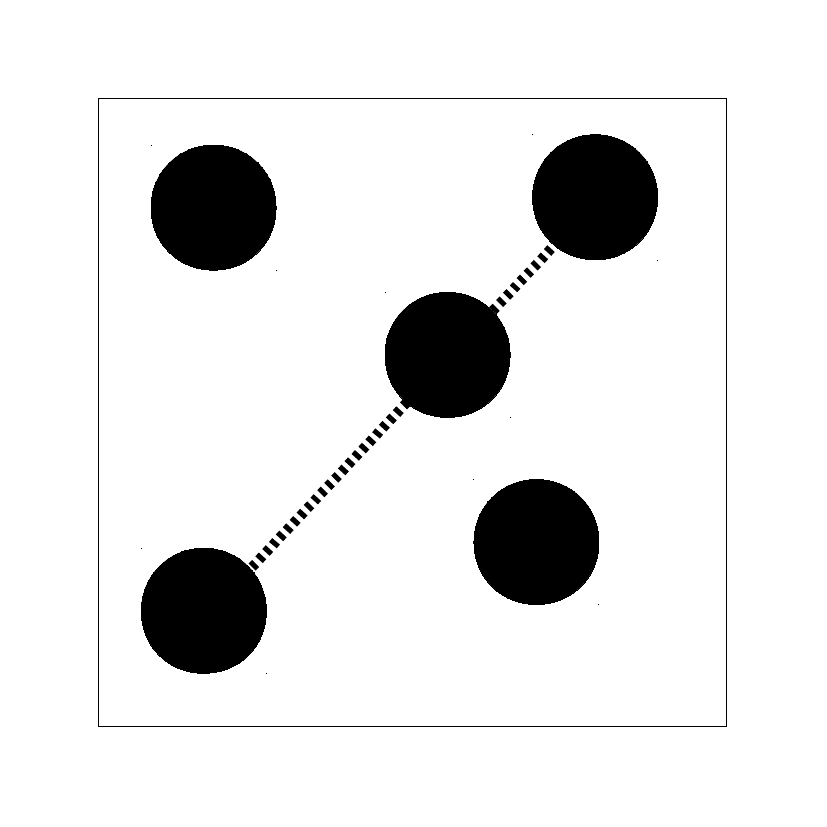}}} }

\newcommand{\hedgehog}{\vcenter{\hbox{\includegraphics[scale = 0.06]{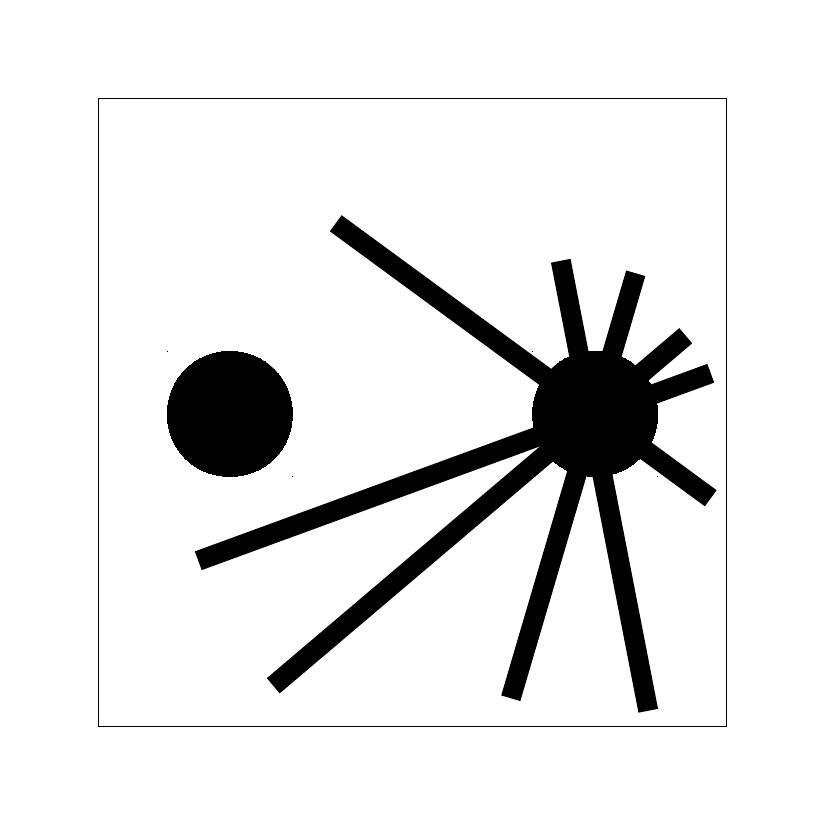}}} }

\newcommand{\fourcam}{\vcenter{\hbox{\includegraphics[scale = 0.06]{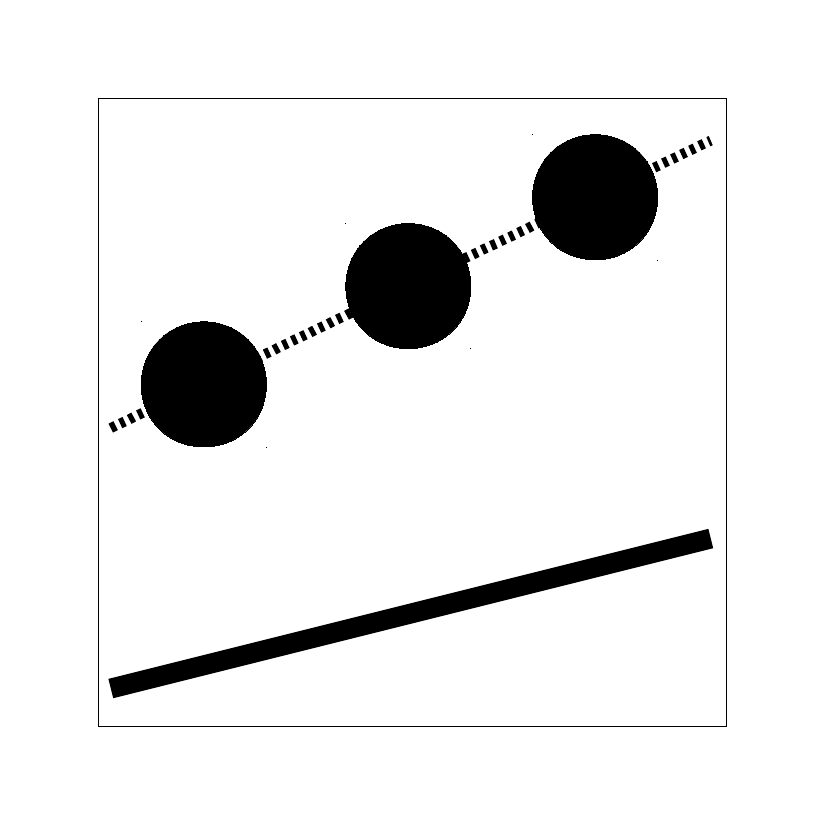}}} }

\title{
Galois/monodromy groups for decomposing\\
minimal problems in 3D reconstruction
}
\author{
Timothy Duff,\thanks{TD: Georgia Institute of Technology. tduff3@gatech.edu}\\
Viktor Korotynskiy,\thanks{Czech Institute of Informatics, Robotics and Cybernetics, Czech Technical University in Prague, CZ. Viktor.Korotynskiy@cvut.cz, pajdla@cvut.cz}\\
Tomas Pajdla\footnotemark[2]
\\
Margaret H. Regan\thanks{MHR: Duke University. mregan@math.duke.edu}
}

\date{\today}
\begin{document}

\maketitle

\begin{abstract}
\noindent
We consider Galois/monodromy groups arising in computer vision applications, with a view towards building more efficient polynomial solvers.
The Galois/monodromy group allows us to decide when a given problem decomposes into algebraic subproblems, and whether or not it has any symmetries.
Tools from numerical algebraic geometry and computational group theory allow us to apply this framework to classical and novel reconstruction problems.
We consider three classical cases---3-point absolute pose, 5-point relative pose, and 4-point homography estimation for calibrated cameras---where the decomposition and symmetries may be naturally understood in terms of the Galois/monodromy group.
We then show how our framework can be applied to novel problems from absolute and relative pose estimation.
For instance, we discover new symmetries for absolute pose problems involving mixtures of point and line features.
We also describe a problem of estimating a pair of calibrated homographies between three images.
For this problem of degree $64$, we can reduce the degree to $16,$ the latter better reflecting the intrinsic difficulty of algebraically solving the problem. 
As a byproduct, we obtain new constraints on compatible homographies, which may be of independent interest. 
\end{abstract}



\section{Introduction}

Estimating 3D geometry from data~\cite{HZ-2003,schoenberger2016sfm,snavely2008modeling} in one or more images is a reoccuring problem in subjects like computer vision, robotics, and photogrammetry.
The exact formulations of these problems vary considerably, depending on factors like the model of image formation that is being considered or what data are available.
A common element in many of these geometric pose estimation problems is the presence of polynomial constraints, depending on both the data and the unknown quantities to be estimated.
Considerable emphasis has been placed on \deff{minimal problems} which are usually well-posed in the sense that exact solutions exist for generic data.
A thorough understanding of these problems is not only theoretically appealing, but has practical consequences.
Minimal problems and their solvers~\cite{AgarwalLST17,Barath-CVPR-2018,Barath-TIP-2018,Barath-CVPR-2017,Byrod-ECCV-2008,DBLP:conf/eccv/CamposecoSP16,PLMP,PL1P,Elqursh-CVPR-2011,Ricardo,Hartley-PAMI-2012,Kileel-MPCTV-2016,DBLP:conf/eccv/KneipSP12,Kuang-ICCV-2013,kuang-astrom-2espc2-13,kukelova2008automatic,larsson2017efficient,Larsson-Saturated-ICCV-2017,larsson2017making,Miraldo-ECCV-2018,mirzaei2011optimal,Nister,DBLP:conf/cvpr/RamalingamS08,SalaunMM-ECCV-2016,saurer2015minimal,Stewenius-ISPRS-2006,ventura2015efficient} play an outsized role in RANSAC-based estimation, introduced to the computer vision community by Fischler and Bolles in 1981~\cite{RANSAC} and later developed to a very efficient and robust estimation method~\cite{Raguram-USAC-PAMI-2013}.
Using a robust minimal solver as a subroutine within the RANSAC loop allows for fewer iterations when compared to other techniques (eg.~linear algorithms~\cite{8pt}) that require more measurements and often ignore the underlying polynomial constraints.
However, the algebraic complexity of a minimal problem as measured by its \deff{degree} has been observed as a limiting factor for minimal solvers, despite their many successes in practice. 

Towards controlling this algebraic complexity, the question of how to detect and exploit symmetries when solving polynomial systems of equations has attracted interest in recent literature in computer vision~\cite{Ask12,LarssonSymmetries}.
The techniques developed there can be understood in the context of linear representation theory, and can be traced back to pioneering work in computer algebra~\cite{Corless,gatermann90}.
As noted in~\cite[Sec 4]{LarssonSymmetries}, these techniques may be difficult to apply without some problem-specific knowledge.
Moreover, many problems have \emph{nonlinear} symmetries.
A famous example is the problem of 5-point relative pose estimation, or simply the 5-point problem. 
The set of solutions to this problem is invariant under a nonlinear symmetry known as the twisted pair.
\begin{example}
\label{ex:5pp}
Two cameras view some number of points in the world (ie.~3-dimensional space.)
Each camera is modeled via perspective projection.
When the cameras are calibrated~\cite[Ch.~6]{HZ-2003}, we may assume that the two camera frames differ by a rotation $\R$ and a translation $\tran.$
\begin{figure}[H]
\begin{center}
\def\svgwidth{0.8\columnwidth}
\begingroup%
  \makeatletter%
  \providecommand\color[2][]{%
    \errmessage{(Inkscape) Color is used for the text in Inkscape, but the package 'color.sty' is not loaded}%
    \renewcommand\color[2][]{}%
  }%
  \providecommand\transparent[1]{%
    \errmessage{(Inkscape) Transparency is used (non-zero) for the text in Inkscape, but the package 'transparent.sty' is not loaded}%
    \renewcommand\transparent[1]{}%
  }%
  \providecommand\rotatebox[2]{#2}%
  \newcommand*\fsize{\dimexpr\f@size pt\relax}%
  \newcommand*\lineheight[1]{\fontsize{\fsize}{#1\fsize}\selectfont}%
  \ifx\svgwidth\undefined%
    \setlength{\unitlength}{198.42519685bp}%
    \ifx\svgscale\undefined%
      \relax%
    \else%
      \setlength{\unitlength}{\unitlength * \real{\svgscale}}%
    \fi%
  \else%
    \setlength{\unitlength}{\svgwidth}%
  \fi%
  \global\let\svgwidth\undefined%
  \global\let\svgscale\undefined%
  \makeatother%
  \begin{picture}(1,0.64285714)%
    \lineheight{1}%
    \setlength\tabcolsep{0pt}%
    \put(0,0){\includegraphics[width=\unitlength,page=1]{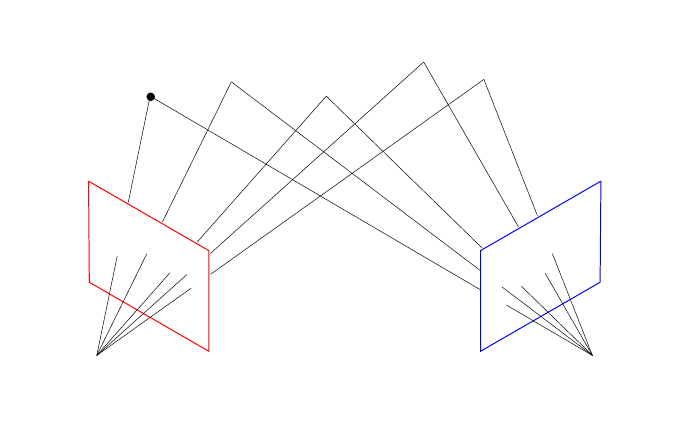}}%
    \put(0.0927371,0.1098441){\color[rgb]{0,0,0}\makebox(0,0)[lt]{\lineheight{1.25}\smash{\begin{tabular}[t]{l}$C_1$\end{tabular}}}}%
    \put(0.8778511,0.11074437){\color[rgb]{0,0,0}\makebox(0,0)[lt]{\lineheight{1.25}\smash{\begin{tabular}[t]{l}$C_2$\end{tabular}}}}%
    \put(0,0){\includegraphics[width=\unitlength,page=2]{5pt.pdf}}%
    \put(0.47719084,0.03421369){\color[rgb]{0,0,0}\makebox(0,0)[lt]{\lineheight{1.25}\smash{\begin{tabular}[t]{l}$\mathbf{R}, \mathbf{t}$\end{tabular}}}}%
    \put(0,0){\includegraphics[width=\unitlength,page=3]{5pt.pdf}}%
  \end{picture}%
\endgroup%

\label{fig:5pt}
\end{center}
\end{figure}
In the five-point problem, we are given the $2$D data of $5$ correspondences between image points $\xx_1 \leftrightarrow \yy_1, \ldots , \xx_5 \leftrightarrow \yy_5$ which are assumed to be images of $5$ world points.
The image points are given in \deff{normalized image coordinates}, meaning that each is a $3\times 1$ vector whose last coordinate equals $1.$
The task is to reconstruct the relative orientation $\cam{\R}{\tran} \in \SERR (3) $ between the two views and each of the five world points as measured by their depths with respect to the first and second camera frames. Writing $\alpha_1 , \ldots , \alpha_5$ for the depths with respect to the first camera and $\beta_1,\ldots , \beta_5$ for the depths with respect to the second camera, the five-point problem becomes a system of polynomial equations and inequations:
\begin{equation} \label{eq:5p_intro}
    \begin{split}
        \R^\top\R = \mathtt{I}, \quad \det \R = 1,
        \\
        \beta_i\yy_i = \R\alpha_i\xx_i + \mathbf{t}, \;\; \alpha_i, \beta_i \neq 0, \quad \forall \,  i = 1,\dots,5.
    \end{split}
\end{equation}
An inherent ambiguity of the five-point problem is that the unknowns $\mathbf{t}, \alpha_1, \ldots , \alpha_5, \beta_1, \ldots , \beta_5$ can only be recovered up to a common scale factor.
If we treat these unknowns as homogeneous coordinates on a $12$-dimensional projective space, then for generic data $\xx_1,\ldots , \xx_5, \yy_1, \ldots , \yy_5,$ there are at most finitely many solutions in $\R, \mathbf{t}, \alpha_1, \ldots, \alpha_5, \beta_1,\ldots , \beta_5$ to the system~\eqref{eq:5p_intro}.
Moreover, if we count solutions over the complex numbers, there are exactly $20$ solutions for generic data in $Z = \left( \CC^2 \times \{1 \}\right)^5 \times \left( \CC^2 \times \{ 1 \} \right)^5.$
The solutions to~\eqref{eq:5p_intro} are naturally identified with the fibers of a branched cover $f:X \to Z$ (see Definition~\ref{def:branchedcover}), 
where $X$ is the \deff{incidence correspondence}
\[
X = \{ \left(\R, (\mathbf{t}, \alpha_1, \ldots , \alpha_5, \beta_1, \ldots , \beta_5), (\xx_1,\ldots , \xx_5, \yy_1 , \ldots , \yy_5) \right) \in \SOCC (3) \times \PP_\CC^{12} \times Z \mid 
\eqref{eq:5p_intro} \text{ holds }
\} .
\]
For most solutions to~\eqref{eq:5p_intro}, the associated \deff{twisted pair} solution is obtained by rotation of the second camera frame $180^\circ$ about the baseline connecting the first and second camera centers.
\begin{figure}[H]
\begin{center}
\def\svgwidth{\columnwidth}
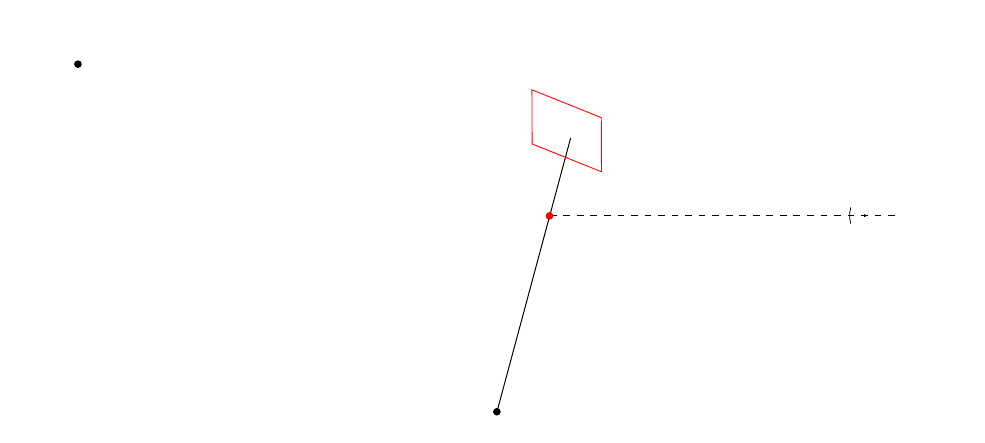
\label{fig:5pt_twisted_pair}
\end{center}
\end{figure}
The twisted pair may be viewed as a rational map $\rat{\Psi}{X}{X},$ given coordinate-wise by
\begin{equation}
\label{eq:twisted-pair}
\begin{split}
\Psi (\R) &= 
\left(
2 \displaystyle\frac{\tran \hfill \tran^\top}{\tran^\top \hfill \tran} \, 
- \I
\right)
\, \R 
\\
\Psi (\tran) &= \tran \\
\Psi (\alpha_i ) &= \displaystyle\frac{-\alpha_i \norm{\tran}^2 }{\norm{\tran}^2 + 2 \, \langle \R^\top \tran , \alpha_i \xx_i \rangle 
}
= \displaystyle\frac{-\alpha_i\norm{\tran}^2}{\norm{\beta_i\yy_i}^2 - \norm{\alpha_i\xx_i}^2}
\\
\Psi (\beta_i ) 
&= 
\displaystyle\frac{\beta_i \norm{\tran}^2 }{\norm{\tran}^2 + 2 \, \langle \R^\top \tran , \beta_i \xx_i \rangle 
}
= \displaystyle\frac{\beta_i\norm{\tran}^2}{\norm{\beta_i\yy_i}^2 - \norm{\alpha_i\xx_i}^2}\\
\Psi (\xx_1, \ldots, \xx_5, \yy_1, \ldots , \yy_5) &= (\xx_1, \ldots, \xx_5, \yy_1, \ldots , \yy_5) .
\end{split}    
\end{equation}
Here, we use the notation $\langle ,  \rangle$ and $\norm{\cdot }^2$ for the complex quadratic forms $\langle \bs{a}, \bs{b} \rangle= a_1b_1 + a_2 b_2 + a_3 b_3,$ $\norm{\bs{a}}^2 = \langle \bs{a}, \bs{a} \rangle $ which restrict to the usual norm and inner product on $\RR^3.$
We note that $\Psi$ is undefined whenever $\bs{t}\in \PP^2$ is an \deff{isotropic vector} satisfying $\norm{\bs{t} }^2 =0,$ and whenever $\norm{\alpha_i\xx_i}^2 = \norm{\beta_i\yy_i}^2$ for some $i = 1,\dots,5$. 
The second condition can be understood geometrically: if the camera centers and the world point $\mathbf{X} = \alpha \xx$ form an isosceles triangle with base $\norm{\tran}$, then, after rotating the second camera, the rays which join camera centers to the respective image points will become parallel.

One can check (e.g.~\cite[p.~20]{Maybank}) that the set of solutions to equations~\eqref{eq:5p_intro} is left invariant under application of $\Psi.$
In other words, we have an equality of mappings $f \circ \Psi = f$ wherever $\Psi$ is defined. 
The map $\Psi $ is a deck transformation
(see Definition~\ref{def:deck}) of the branched cover $f.$
\end{example}
The twisted pair of Example~\eqref{eq:5p_intro} is a well-known construction.
It is not easy to determine that such a symmetry exists just from staring at Equations~\eqref{eq:5p_intro}.
However, the existence of such a symmetry can be easily decided after computing the Galois/monodromy group of $f,$ as we make explicit in Proposition~\ref{prop:deck-centralizer}.
The state-of-the-art method~\cite{Nister} for solving the five-point problem is based on \deff{decomposing} the branched cover $f$ in terms of the \deff{essential matrix} (see Examples~\ref{ex:resolve-twisted}, \ref{ex:deck-resolve-twisted} and Section~\ref{subsec:5pp}).
In general, decomposability can be detected after computing the Galois/monodromy group, by Proposition~\ref{prop:decomposable_iff_imprimitive}.

Working within the framework of branched covers and Galois/monodromy groups allows us to probe the structure of minimal problems in ways that the previously mentioned works cannot.
As is well-known (see Proposition~\ref{prop:galoismonodromy}), this group can be understood both algebraically (``Galois'') and topologically (``monodromy'').
The algebraic point of view was pursued by Hartley, Nist\'{e}r, and Stew\'{e}nius~\cite{NisterHartley}, who computed Galois groups symbolically to show that certain formulations of minimal problems were degree-optimal.
In our paper, the topological point of view is somewhat more relevant, since we can compute the monodromy action for problems of potentially large degree using numerical homotopy continuation methods, as  implemented in several software packages~\cite{Bertini,HCJL,MonodromySolver}.

Computing the Galois/monodromy group allows us to decide decomposability or the existence of deck transformations by reduction to standard algorithms in computational group theory~\cite{holt}. These algorithms are implemented in computer algebra systems such as GAP~\cite{GAP}, which was an essential ingredient in discovering our main results.
We find frequently that many minimal problems do not decompose---in this case, the Galois/monodromy group tells us that our formulation is optimal in a precise sense.
However, we find in several cases that minimal problems previously considered in the literature do decompose.
We describe this in the context of problems where this phenomenon is already well-understood (such as the five-point problem), and also in several cases where some decomposition or symmetry was not previously noticed.

Many of our results are computational, and thus may fall short of the conventional standard for proof.
We use the term ``Result'' instead of ``Theorem'' in these cases.

\begin{remark}
\label{remark:symbolicGalois}
Our Galois groups are geometric, over the field $\CC,$ rather than arithmetic, over the field $\QQ .$
In general, the geometric Galois group is a normal subgroup of the arithmetic Galois group.
The two can be different---for instance, the arithmetic Galois group for Example~\ref{ex:x^3} is $S_3.$
However, we expect these two groups to be the same for most problems considered here.
Several problems (such as the absolute pose problems appearing in Section~\ref{sec:abspose}) are well-within the range of symbolic Galois group computation.
For problems of higher degree appearing in Section~\ref{sec:rel-pose},
an alternative route for recognizing the frequently occuring symmetric groups would use heuristics based on the Chebotarev density theorem as in~\cite{del2019classification} after computing certain lexicographic Gr\"{o}bner bases.
We do not pursue this route.
\end{remark}

For readers with a background in computer vision, we suggest starting with the familiar examples in Sections~\ref{sec:abspose} and~\ref{sec:rel-pose}, referring back to the theory and examples in Section~\ref{sec:background} as needed.
The mathematics needed to fully understand our paper is presented in Section~\ref{subsec:branch}.
We include proofs of several standard facts in hope that our paper is accessible to a broad audience.
We give a brief overview of numerical methods for computing Galois/monodromy groups in Section~\ref{subsec:num_methods}. 
In Section~\ref{sec:abspose}, we investigate the Galois/monodromy groups of absolute pose problems involving points and lines, where the task is to compute cameras from correspondence data between the world and an image.
In particular, we describe newly discovered symmetries for problems involving a mixture of points and lines.
Section~\ref{sec:rel-pose} considers relative pose problems, where the task is to compute the transformation between camera frames given correspondence data between images.
We describe decompositions for the five-point problem in Section~\ref{subsec:5pp}, the four-point calibrated homography problem for two views in Section~\ref{subsec:2viewhomography}, and a novel minimal problem involving two calibrated homographies between three images in Section~\ref{subsec:3viewhomography}. 
We give a conclusion and outlook in Section~\ref{sec:outlook}.

\section{Background}
\label{sec:background}

\subsection{Branched covers and monodromy groups}
\label{subsec:branch}
\begin{defn}
\label{def:branchedcover}
A \deff{branched cover} is a dominant, rational map $\ratnoname{X}{Z},$ where $X$ and $Z$ are irreducible algebraic varieties over $\CC$ of the same dimension.
\end{defn}
We emphasize several consequences of this definition.
Most importantly, it follows that for generic (and hence \emph{almost all}) data $z\in Z,$ the fiber over $z,$ denoted $X_z ,$ is a nonempty, finite set.
Second is the assumption of irreducibility, which implies that the monodromy group acts transitively.
In principal, we can always reduce to the case of an irreducible variety by considering the irreducible components of an arbitrary variety.
The many examples we consider show that irreducibility is a natural assumption.
Finally, although many of the branched covers we consider are actually regular maps, the maps appearing as deck transformations or in a factorization need not be defined on the same domain as the branched cover, making it more natural to work with rational maps.
We say that $X$ and $Z$ are the \deff{total space} and \deff{base} of the branched cover, respectively.
The reader may safely assume that all varieties are quasiprojective. 

Pulling back rational functions from $Z$ to $X$ lets us identify $\CC (Z)$ with a subfield of $\CC (X)$.
Since $\CC (X)$ and $\CC (Z)$ have the same transcendence degree over $\CC ,$ the field extension $\CC(X) / \CC(Z)$ is finite.
The \deff{degree} of the map $\ratnoname{X}{Z}$ may be defined as the degree of this field extension.
We write $\deg (X/Z)$ for this quantity, since the map $\ratnoname{X}{Z}$ is usually clear from context.
We say that a nonempty Zariski-open $U\subset Z$ is a \deff{regular locus} for $\ratnoname{X}{Z}$ if $U \cap Z_{\sing } = \emptyset$ and if for all $z\in U$ the cardinality of the fiber $X_z := f^{-1}(z)$ is equal to the degree of the map.
The existence of such a $U$ follows from basic results in algebraic geometry~\cite[cf.~pp.~142]{Shaf}.

We now recall the monodromy action on the fibers of of a degree-$d$ branched cover $\rat{f}{X}{Z}.$
We omit many details which may be found in several excellent introductory references~\cite{Hatcher,Mir95,Zol06}.
Fix a regular locus $U$ and a basepoint $z\in U,$ and write  $X_z = \{x_1,\dots,x_d\}.$ 
A \deff{loop} based at $z$ is a continuous map $\gamma : [0,1] \to U$ which satisfies $\gamma (0) = \gamma (1) = z.$
For each $x_i,$ there exists a unique \deff{lift}  $\widetilde{\gamma}_i : [0, 1] \to f^{-1} (U)$ satisfying $\gamma = f\circ \widetilde{\gamma}_i $ and $\gamma (0) = x_i.$ 
This fact from topology is known as the \deff{unique path-lifting property} (see eg.~\cite[Proposition 1.30]{Hatcher}).
The lifts based at each of the points $x_1,\ldots , x_d$ determine a permutation of the fiber, $\sigma_{\gamma}: X_z \to X_z$, which may be written in two-line notation as
\begin{equation}
\label{eq:permutation-lift}
\sigma_{\gamma } = 
\left(\begin{array}{ccccc}
x_1 & x_2 & x_3 & \cdots & x_d \\
\widetilde{\gamma}_1(1) & \widetilde{\gamma}_2(1) & \widetilde{\gamma}_3(1) & \cdots & \widetilde{\gamma}_d(1) 
\end{array}\right). 
\end{equation}

\begin{remark}
\label{remark:groupTheory}
We use standard notation from group theory: $\sym (X)$ is the symmetric group of all permutations from a finite set $X$ to itself, $S_n, \, A_n, $ denote symmetric and alternating groups acting on letters $[n] = \{ 1, \ldots , n \}$ (hence $\sym ([n]) = S_n$), and $C_n$ denotes a cyclic group of order $n.$ 
At several points, we must distinguish between abstract groups and the way they act on sets.
For instance, the usual action of $S_4$ on $[4]$ is not equivalent to the action of $S_4 \hookrightarrow S_6$ on the $6 = \binom{4}{2}$ unordered pairs in $[4]$.
The latter group may also described as $S_2 \wr S_3 \cap A_6.$ 
Here, $\wr$ denotes the \deff{wreath product}; the group $S_2 \wr S_3 $ may be realized as the subgroup of $S_6$ preserving the partition $[6] = [2] \cup \{3,4\} \cup \{5,6\}.$
In general, the wreath product $S_m \wr S_n$ is a semidirect product $\left(S_m\right)^n \rtimes S_n,$ and is usually equipped with an action on the Cartesian product $[n] \times [m].$
We refer to~\cite[Chapter 7]{Rotman} for several useful facts about the wreath product.
The group $S_2 \wr S_3 \cap A_6$ will reappear in Section~\ref{subsec:2viewhomography} on homography estimation.
\end{remark}

The permutation $\sigma_\gamma $ is independent of the homotopy class of $\gamma $ in $U,$ from which one obtains a homomorphism from the fundamental group $\pi_1 (U , z) $ into the symmetric group $\sym ( X_z)$:
\begin{equation}
  \label{eq:monodromyRep}
  \begin{split}
  \rho : \pi_1 (U , z) &\to \sym (X_z)\\
       [\gamma ] &\mapsto \sigma_{\gamma }.
       \end{split}
\end{equation}
The map $\rho $ in Equation~\ref{eq:monodromyRep} is known as the \deff{monodromy representation.}
The image of this homomorphism is a permutation group which acts transitively on the fiber $X_z.$
It is called the \deff{monodromy group,} and will be denoted by $\mon (X/Z; U, z),$ or $\mon (X/Z;z),$ or simply $\mon (X/Z).$
The latter notation, and our preferred terminology of \deff{Galois/monodromy group}, is justified by Proposition~\ref{prop:galoismonodromy}.
We write $\galclo{\CC (X)} / \CC (Z)$ for the Galois closure of $\CC(X) / \CC(Z)$ and $\Gal (X / Z)$ for its Galois group.
\begin{prop}
\label{prop:galoismonodromy}
Let $\ratnoname{X}{Z}$ be a branched cover with regular locus $U,$ and fix a basepoint $z\in U.$
Then the $\Gal (X/Z)$ and $\mon (X/Z ; U, z)$ are isomorphic as permutation groups.
In particular, $\mon (X/Z; U, z)$ is independent of the choice of $(U,z).$
\end{prop}
A proof of Proposition~\ref{prop:galoismonodromy} is given in~\cite{Harris}.
In this proof, the Galois closure $\galclo{\CC (X)} / \CC (Z)$ is identified as an extension of $\CC (X)$ obtained by adjoining certain germs of functions around points in $X_z.$ 
With this identification, it is then argued (using the Galois correspondence and analytic continuation) that the Galois and monodromy actions on $X_z$ coincide.
\begin{example}
\label{ex:x^3}
Consider
\[
X = \{ (x,\, z)  \in \CC \times \CC \mid x^3 = z \}\]
as a degree-$3$ branched cover over $Z=\CC $ given by $(x,\, z) \mapsto z.$
A regular locus is the punctured complex line $U=\{ z \mid z\ne 0 \}.$
The monodromy group $\mon (X/Z) \cong A_3 = C_3$ acts by cyclic permutation of $X_z=\{ z, \omega \, z , \omega^2 \, z \},$ where $\omega = \exp (2 \pi i / 3).$
Indeed, $\pi_1 (U ;z)$ is generated by the loop $\gamma (t) = e^{2 \pi i t},$ which encircles the branch point $z=0$ and induces the permutation $\sigma_\gamma$ defined by 
\[
\sigma_{\gamma } = 
\left(\begin{array}{ccccc}
z & \omega \, z & \omega^2 \, z   \\
\omega \, z   & \omega^2 \, z & z 
\end{array}\right). 
\]
\end{example}
\begin{defn}
\label{def:bir}
Two branched covers, $\ratnoname{X_1}{Z_1}$ and $\ratnoname{X_2}{Z_2}$, are \deff{birationally equivalent} if there exist birational maps $ \ratnoname{X_1}{X_2}$ and $\ratnoname{Z_1}{Z_2}$ such that the following diagram commutes:
\begin{center}
\begin{tikzcd}
X_1 \arrow[d, dashed] \arrow[r, dashed] & X_2\arrow[d,dashed] \\ Z_1 \arrow[r, dashed ] & Z_2 .
\end{tikzcd}
\end{center}
\end{defn}
\begin{prop}
\label{prop:bir}
The Galois/monodromy group of a branched cover is a birational invariant.
\end{prop}
\begin{proof}
This follows easily from Proposition~\ref{prop:galoismonodromy}, since a birational equivalence of branched covers induces an isomorphism of field extensions.
A more topological proof is also possible.
Let us write $\rat{\Psi}{X_1}{X_2},$ $\rat{\psi}{Z_1}{Z_2}$ for the maps appearing in Definition~\ref{def:bir}. 
Then, for suitable regular loci $U_1 \subset Z_1, U_2 \subset Z_2,$
there is an isomorphism $\mon (X_1 / Z ; U_1, z) \cong \mon (X_2 / Z ; U_2 , \Psi (z))$ which may be defined by identifying, for each $\gamma : [0,1] \to U_1$ based at $z,$  the lifts $\widetilde{\gamma}_1, \ldots , \widetilde{\gamma }_d$ in $X_1$ with the lifts $\Psi \circ \widetilde{\gamma}_1, \ldots , \Psi \circ \widetilde{\gamma }_d$ in $X_2.$
\end{proof}

\begin{example}
\label{ex:resolve-twisted}
Consider the regular branched cover $\SOCC (3) \times \PP^2 \to \mathcal{E}$ given by 
\begin{equation}
\label{eq:twistedPairMap}
(\R, \mathbf{t} ) \mapsto [\tran]_\times  \R ,
\end{equation}
where for $\mathbf{t}=[t_1:t_2:t_3]$ we let
\[
[\tran]_\times = 
\left(
\begin{smallmatrix}
0 & -t_3 & t_2\\
t_3 & 0 & -t_1\\
-t_2 & t_1 & 0
\end{smallmatrix}
\right)
\]
denote a $3\times 3$ matrix that represents, up to scale, taking the cross product with $\tran$, and 
$\mathcal{E}$ denotes the variety of \deff{essential matrices}:
\begin{equation}
    \label{eq:Vess}
    \mathcal{E} = \{ \E \in \PP (\CC^{3\times 3}) \mid \displaystyle\det \E =0, \, \, \E \E^\top \E -  \displaystyle\frac{1}{2} \, \tr (\E \E^\top)  \E = 0 \}.
\end{equation}
A regular locus is given by $U\subset \mathcal{E}$ such that all $E \in U$ have rank $2$ and the kernel of $\E$ is not spanned by an isotropic vector.
A birationally equivalent branched cover was constructed in~\cite{2Hilbert,van2019functorial}, where the authors construct moduli spaces obtained by letting the \emph{absolute conic}~\cite{HZ-2003} degenerate to a double line.
Explicitly, the branched cover $X \to \mathcal{E}$ is given by 
\[
X = \{ \left( [a_0:a_1:a_2:a_3],  [b_0:b_1:b_2:b_3] \right) \in \PP^3 \times \PP^3 \mid a_0 b_0 + a_1 b_1 + a_2 b_2 + a_3 b_3 = 0 \}
\]
\[
\left( [\mathbf{a}],  [\mathbf{b}] \right) \mapsto \left(\begin{smallmatrix}
     {a}_{0}{b}_{0}-{a}_{1}{b}_{1}-{a}_{2}{b}_{2}+{a}_{3}{b}_{3}&{a}_{1}{b}_{0}+{a}_{0}{b}_{1}+{a}_{3}{b}_{2}+{a}_{2}{b}_{3}&{a}_{2}{b}_{0}-{a}_{3}{b}_{1}+{a}_{0}{b}_{2}-{a}_{1}{b}_{3}\\
     {a}_{1}{b}_{0}+{a}_{0}{b}_{1}-{a}_{3}{b}_{2}-{a}_{2}{b}_{3}&-{a}_{0}{b}_{0}+{a}_{1}{b}_{1}-{a}_{2}{b}_{2}+{a}_{3}{b}_{3}&{a}_{3}{b}_{0}+{a}_{2}{b}_{1}+{a}_{1}{b}_{2}+{a}_{0}{b}_{3}\\
     {a}_{2}{b}_{0}+{a}_{3}{b}_{1}+{a}_{0}{b}_{2}+{a}_{1}{b}_{3}&-{a}_{3}{b}_{0}+{a}_{2}{b}_{1}+{a}_{1}{b}_{2}-{a}_{0}{b}_{3}&-{a}_{0}{b}_{0}-{a}_{1}{b}_{1}+{a}_{2}{b}_{2}+{a}_{3}{b}_{3}\\
     \end{smallmatrix}\right),
\]
and there exists a birational equivalence
\begin{center}
\begin{tikzcd}
X \arrow[d] \arrow[r, dashed] & \SOCC (3) \times \PP^2 \arrow[d] \\ \Ess \arrow[r] & \Ess .
\end{tikzcd}
\end{center}
where the bottom map is the identity. 
The top map may be given by
\[
\left( [\bs{a}],  [\bs{b}] \right) \mapsto 
\left(
(a_3 \, I + [\mathbf{a}]_\times) ([\mathbf{a}]_\times - a_3 \, I)^{-1}
      ,
      \,
 \left(\begin{smallmatrix}
      2\,{a}_{1}{b}_{0}-2\,{a}_{0}{b}_{1}+2\,{a}_{3}{b}_{2}-2\,{a}_{2}{b}_{3}\\
      2\,{a}_{2}{b}_{0}-2\,{a}_{3}{b}_{1}-2\,{a}_{0}{b}_{2}+2\,{a}_{1}{b}_{3}\\
      2\,{a}_{3}{b}_{0}+2\,{a}_{2}{b}_{1}-2\,{a}_{1}{b}_{2}-2\,{a}_{0}{b}_{3}\\
      \end{smallmatrix}\right)      
\right),
\]
where now
\[
[\mathbf{a}]_\times = 
\left(\begin{smallmatrix}
      0&{a}_{0}&{a}_{1}\\
      -{a}_{0}&0&{a}_{2}\\
      -{a}_{1}&-{a}_{2}&0\\
      \end{smallmatrix}\right).
\]
The map $\ratnoname{X}{\SOCC (3)}$ is undefined for $([\bs{a}], [\bs{b}])$ such that $\bs{a}$ lies on the isotropic quadric $\norm{\bs{a}}^2 = 0$ in $\PP^3.$ 
In~\cite{van2019functorial}, it was observed that a regular locus of $X \to \mathcal{E}$ is simply given by any $U\subset \mathcal{E}$ such that all $\E \in U$ have rank $2.$
Both branched covers have degree $2,$ and $\mon (X/ \mathcal{E})$ is the symmetric group $S_2$ acting on two letters.
We note that the map $\ratnoname{X}{\SOCC (3) \times \PP^2}$ is closely related to dual quaternions and Study coordinates for $\SE (3)$~\cite{DBLP:journals/ijrr/Daniilidis99}.
\end{example}
A \deff{factorization} of a branched cover $\ratnoname{X}{Z}$ is a commutative diagram
\begin{equation}
\label{eq:factorization-diagram}
\begin{tikzcd}
X \arrow[rd, dashed] \arrow[r,  dashed] & Y\arrow[d,dashed] \\& Z
\end{tikzcd}
\end{equation}
such that $\ratnoname{X}{Y}$ and $\ratnoname{Y}{Z}$ are branched covers. 
If $\deg (X/Y)$ and $\deg (Y/Z)$ are both strictly less than $\deg (X/Z),$ we say that the factorization is \deff{proper} and that the branched cover $\ratnoname{X}{Z}$ is \deff{decomposable}.
Otherwise, $\ratnoname{X}{Z}$ is \deff{indecomposable.} Proposition~\ref{prop:decomposable_iff_imprimitive}
implies that the decomposability of a branched cover can be determined from the Galois/monodromy group alone.
We recall that a \deff{block system} for the monodromy action
$\mon (X/Z) \acts X_z = \{ x_1 . \ldots x_d \},$ is a partition of $X_z = B_1 \cup \cdots \cup B_k,$ comprised of equally-sized blocks $B_1 , \ldots , B_k,$ which is preserved in the sense that blocks are always mapped to blocks under the group action.
The block systems associated to the action form a lattice under refinement, whose respective maximum and minimum elements are $\{ X_z \}$ and $\{ \{ x_1\} , \ldots , \{ x_d \} \}.$
If any other block systems exist, then $\mon (X/Z)$ is said to be \deff{imprimitive}, and otherwise it is \deff{primitive}.

Given a factorization~\eqref{eq:factorization-diagram}, we have $\deg (X/Z) = \deg (X/Y) \, \deg (Y/Z),$ and a partition
\begin{equation}
\label{eq:partition}
X_z = X_{y_1} \cup \cdots \cup X_{y_k} \end{equation}
with $k=\deg (Y/Z).$
The proof of Proposition~\ref{prop:factor} below shows that this is a block system for the monodromy action with blocks of size $\deg (X/Y).$
Conversely, imprimitivity implies decomposability.

\begin{proposition} \label{prop:decomposable_iff_imprimitive}
A branched cover is decomposable if and only if $\mon(X/Z)$ is imprimitive. 
\end{proposition}

Proposition~\ref{prop:decomposable_iff_imprimitive} dates back to the work of Ritt~\cite{Ritt1}, who characterized the possible decompositions of branched covers  $\CC \ni x \mapsto p(x) \in \CC$ given by a univariate polynomial $p.$
A Galois-theoretic proof of Proposition~\ref{prop:decomposable_iff_imprimitive} may be found, for instance, in~\cite{Brysiewicz}.
If we know $\mon (X/Z),$ it is also possible to identify $\mon (X/Y)$ and $\mon (Y/Z)$ occuring in the factorization~\eqref{eq:factorization-diagram}, as the next proposition shows.
\begin{prop}
\label{prop:factor}
Consider a factorization of branched cover as in Equation~\eqref{eq:factorization-diagram}.
For fixed generic $z\in Z,$ partition $X_z$ as in Equation~\ref{eq:partition}.
The action $\mon (X/Z) \acts X_z$ induces two other group actions which are equivalent to the monodromy groups of the individual factors:
\begin{itemize}
\item[1)] action on blocks: $\mon (X/Z) \acts \{ X_{y_1}, \ldots , X_{y_k} \},$ which is equivalent to $\mon (Y/Z).$
\item[2)] action on a single block: $\mon (X/Z)_{X_{y}} \acts X_{y},$ where $\mon (X/Z)_{X_{y}}$ denotes the stabilizer of the set $X_{y}$ under the action by $\mon (X/Z)\acts X_z.$
This is equivalent to $\mon (X/Y),$ and thus independent of the choice $y\in Y_z.$
\end{itemize}
\end{prop}
\begin{proof}
1) For each $\sigma_\gamma \in \mon (X/Z),$ there is an induced permutation of the blocks:
\[
\widetilde{\sigma_\gamma } =
\left(\begin{array}{ccccc}
X_{y_1} & \cdots & X_{y_k} \\
\sigma_\gamma (X_{y_1}) & \cdots & \sigma_{\gamma} (X_{y_k})
\end{array}\right).
\]
Indeed, suppose that $x, x ' \in X_{y_i}$ are such that $\sigma_\gamma (x) \in X_{y_j}, \sigma_\gamma (x') \in X_{y_k},$ and consider the lift $\widetilde{\gamma}: [0,1] \rightarrow Y$ starting at $y_i.$ 
We must have both $\widetilde{\gamma}(1) = y_j$ and $\widetilde{\gamma}(1) = y_k.$ 
Hence $k=j$ by the unique path-lifting property applied to $\ratnoname{Y}{Z}$, showing that $\sigma_\gamma $ preserves the partition into blocks.
In this way we get a group homomorphism
\begin{equation}
\label{eq:hom1}
  \begin{split}
\mon (X/Z) &\to \sym ( \{ X_{y_1} , \ldots , X_{y_k} \} )\\
\sigma_\gamma &\mapsto \widetilde{\sigma_\gamma },
\end{split}
\end{equation}
which represents the action of $\mon (X/Z)$ on the blocks.
Now, there is also an injective group homomorphism
\begin{equation}
\label{eq:hom2}
  \begin{split}
\mon (Y/Z ) &\to \sym ( \{ X_{y_1} , \ldots , X_{y_k} \} )\\
\tau_\gamma &\mapsto 
\left(\begin{array}{ccccc}
X_{y_1} & \cdots & X_{y_k} \\
X_{\tau_\gamma (y_1)} & \cdots & X_{\tau_\gamma (y_k)}
\end{array}\right)
\end{split}
\end{equation}
obtained by restricting the natural isomorphism $\sym (Y_z) \cong \sym (\{ X_{y_1}, \ldots , X_{y_k} \} )$ that identifies a point $y_i \in Y_z$ with its corresponding block $X_{y_i}.$
We wish to show that maps~\eqref{eq:hom1} and~\eqref{eq:hom2} have the same image.
This follows easily if we restrict $\gamma $ in both maps to be loops 
contained in a regular locus for $\ratnoname{X}{Z}$:
the lifts of $\gamma $ to $Y$ (which, by our restriction, also lift to $X$) determine the corresponding permutation of blocks in $X$, and vice-versa.
Indeed, we have $\sigma_\gamma  (X_y)  = X_{\tau_\gamma (y)}$ for any $y\in Y_z$.
To see this, it is enough to show one set is contained in the other.
A point $x\in \sigma_\gamma (X_y)$ is the endpoint of some lift of $\gamma $ to $X.$
The image of this lift in $Y$ is itself a lift $\wt{\gamma } : [0,1] \to Y$ with $\wt{\gamma} (0) = y$---hence $\wt{\gamma } (1) = \tau_\gamma (y),$ and the endpoint of our original lift $x$ is in $X_{\tau_\gamma (y)}.$\\\\
2) The proof amounts to showing that a loop $\gamma$ in $Z$ lifts to a loop in $Y$ if and only if $\sigma_\gamma  \in \mon (X/Z)$ stabilizes each of the blocks.
As in the previous part, this only true if we consider loops in a suitably small regular locus $U\subset Z.$
It suffices to take $U$ contained in a regular locus for $\ratnoname{X}{Z}$ and whose preimage in $Y$ is a regular locus for $\ratnoname{X}{Y}.$
\end{proof}

Proposition~\ref{prop:decomposable_iff_imprimitive} shows that an arbitrary branched cover $\ratnoname{X}{Z}$ factors as a composition of indecomposable branched covers:
\begin{equation}
\label{eq:decomposition}
X = Y_0 \dashrightarrow Y_1 \dashrightarrow \cdots \dashrightarrow Y_{k-1} \dashrightarrow Y_k = Z.
\end{equation}
Such a factorization corresponds to a maximal chain in the lattice of block systems, and the associated degrees can be read off from the block sizes.
Equivalently, for any $x\in X_z,$ a maximal chain in the lattice of blocks of $\mon (X/Z)$ corresponds to a chain of subgroups that contain the stabilizer $\mon (X/Z)_x$ (cf.~\cite[proof of Theorem 9.15]{Rotman})
\[
\mon (X/Z)_x  = G_0 \subset G_{1} \subset \cdots \subset G_{k-1} \subset G_k = \mon (X/Z),
\]
and we have $\deg (Y_i / Y_{i+1}) = [G_{i+1} : G_i ]$ for $i=0, \, \ldots , \, k-1.$
The decomposition~\eqref{eq:decomposition} is not unique.
In fact, as Ritt already understood~\cite[p.~53]{Ritt1}, there are many examples where even the multi-set of degrees $\deg (Y_i / Y_{i+1})$ is not unique.
See~\cite[Example 25]{Gutierrez} for one such explicit example.

Finally, we define and carefully study the deck transformations of a branched cover, of which the twisted pair symmetry from the introduction is a special case. 

\begin{defn}
\label{def:deck}
A birational equivalence from a branched cover to itself which fixes the base is called a \deff{deck transformation.}
Explicitly, for $\rat{f}{X}{Z}$ a deck transformation $\rat{\Psi}{X}{X}$ must satisfy $f \circ \Psi = f$ whenever both maps are defined. 
The deck transformations form a group under composition which acts on a generic fiber $X_z.$
The deck transformation group can be naturally identified with the automorphisms of $\CC (X)$ which fix $\CC (Z),$ denoted $\Aut (X/Z).$
\end{defn}

Analogously to decomposability, Proposition~\ref{prop:deck-centralizer} shows that the existence of a nontrivial deck transformation can be decided from the Galois/monodromy group alone.
This turns out to be stronger than decomposability in general.
We learned of Proposition~\ref{prop:deck-centralizer} from the sources~\cite{awtrey,Cukierman}.
Since it seems less well-known outside of the literature on Galois/monodromy groups, we give a self-contained proof.
In topology, a deck transformation of a covering map $f$ can be any continuous function $\Psi$ satisfying $f\circ \Psi = f.$
Our proof of Proposition~\ref{prop:deck-centralizer} reveals that, for a rational branched cover $f $ with regular locus $U,$ the deck transformations of $f_{|f^{-1}(U) } $ in the topological sense are always rational maps in the sense of Definition~\ref{def:deck}.
Before giving the proof, we first consider three illustrative examples.

\begin{example}
\label{ex:deck-quad}
Let $X = \VV(x^2+ax+b) \subset \CC^3$, $Z = \CC^2$ and $f : X \to Z$ be the degree-$2$ branched cover defined by coordinate projection $f(x,a,b) = (a,b).$ 
The deck transformation defined by $\Psi (x,a,b) = (-x - a, a, b)$ acts on a generic fiber $X_{(a,b)}$ by permuting the two roots of the quadratic equation  $x^2 + a x + b=0.$
\end{example}

\begin{example}
\label{ex:deck-pnp}
Ask et al.~\cite{Ask12} define a polynomial system $F(\bs{x})$ with $p$-fold symmetry to be such that $F(\bs{x}) =0$ implies $F( \omega \, \bs{x})=0 $ whenever $\omega $ is a $p$-th root of unity.
For example, the equations
\begin{figure}
    \centering




\includegraphics[scale = 0.4]{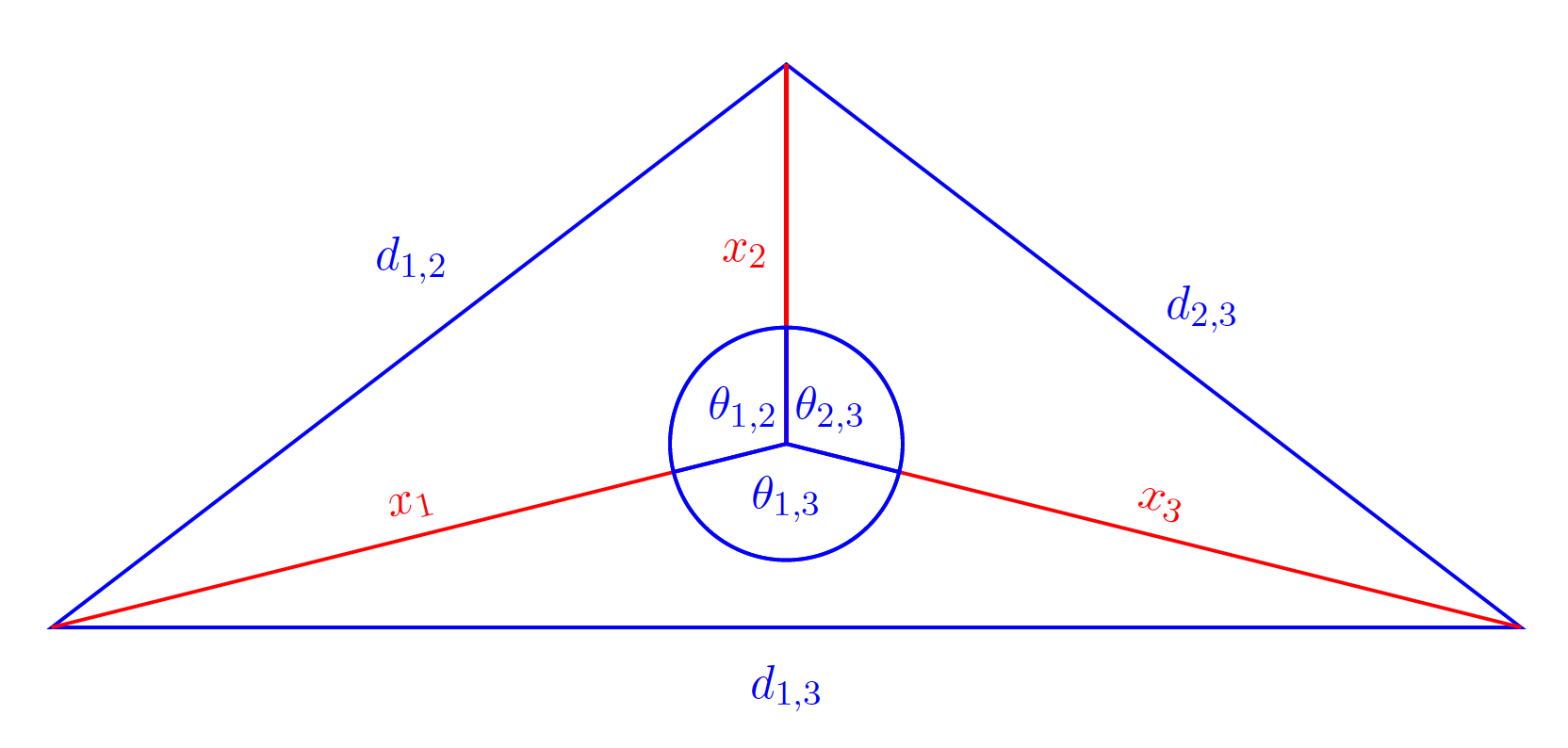}
\caption{Frontal view of the P3P problem: \red{$x_1,x_2,x_3$} are unknown.
}
\label{fig:p3p}
\end{figure}

\begin{equation}
\label{eq:p3p}
\begin{split}
f_{1,2} = x_1^2 + x_2^2 - c_{1,2} x_1 x_2 - d_{1,2}^2\\
f_{1,3} = x_1^2 + x_3^2 - c_{1,3} x_1 x_3 - d_{1,3}^2\\
f_{2,3} = x_2^2 + x_3^2 - c_{2,3} x_2 x_3 - d_{2,3}^2
\end{split}
\end{equation}
have a $2$-fold sign symmetry: $(x_1,x_2,x_3)\mapsto (-x_1, - x_2, -x_3).$
These equations define the famous Perspective-3-Point problem or \deff{P3P problem}.
Here, each $c_{i,j}$ is equal to $2 \cos \theta_{i,j}$ as in Figure~\ref{fig:p3p}.
Letting $X$ denote the vanishing locus of~\eqref{eq:p3p} in $\CC^{9},$ the coordinate projection onto the space of knowns $\CC^6$ is a branched cover with a deck transformation given by the sign-symmetry.
We will return to the P3P problem in Section~\ref{sec:abspose}.
\\\\
Ask et al.~\cite{Ask12} develop algorithms for detecting and exploiting \emph{partial} $p$-fold symmetries (occurring in only some subset of the variables) in the automatic generation of polynomial solvers.
These methods were generalized by Larsson and  {\AA}str{\"o}m~\cite{LarssonSymmetries} to the case of \emph{weighted} partial-$p$ fold symmetries.
In general, a branched cover with a weighted partial-$p$ fold symmetry will have a deck transformation of order $p,$ degree $e \, p$ for some integer $e ,$ and its Galois/monodromy group will be a subgroup of $C_p \wr S_e.$
\end{example}

\begin{example}
\label{ex:deck-resolve-twisted}

Example~\ref{ex:resolve-twisted} contains two interesting \deff{Galois covers}---the Galois/monodromy group and the deck transformation group are isomorphic.
For $\ratnoname{\SOCC (3) \times \PP^2}{\Ess},$ the action on the fiber applies the twisted pair map as in~\eqref{eq:twisted-pair}.
For $\ratnoname{X}{\Ess},$ the action swaps coordinates $\left([\bs{a}], [\bs{b}]\right) \mapsto \left([\bs{b}], [\bs{a}]\right).$  
\end{example}

\begin{prop}
\label{prop:deck-centralizer}
Let $\ratnoname{X}{Z}$ be a branched cover and fix generic $z\in Z$.
We may identify the deck transformation group with a subgroup of $\sym (X_z)$ by restricting functions to $X_z.$
This permutation group is \emph{equal} to the centralizer of $\mon (X/Z)$ in $\sym (X_z).$
\end{prop}

\begin{proof}
We abbreviate the deck transformation group and centralizer subgroup by $D$ and $C,$ respectively.
We define a map between these groups as follows:
\begin{align*}
    \varphi: D &\rightarrow \sym (X_z) \\
    \Psi &\mapsto \permshort{x_1}{x_d}{\Psi (x_1)}{\Psi (x_d)}
\end{align*}
To prove Proposition~\ref{prop:deck-centralizer}, we verify the following properties of $\varphi $:
\begin{itemize}
    \item[1)] $\varphi$ is a group homomorphism.
    \item[2)] $\varphi $ is injective.
    \item[3)] The image of $\varphi $ is contained in $C.$
    \item[4)] $C$ is contained in the image of $\varphi $---more explicitly, for all $\sigma \in C$ there exists a deck transformation $\Psi_\sigma \in D$ whose restriction to the fiber $X_z$ equals the permutation $\sigma .$
\end{itemize}
Property 1) is straightforward. 
Properties 2) and 3) both follow from the unique path-lifting property.
For instance, if $\Psi (x_i)=x_i$ for $i=1,\ldots , d,$ then for generic $x\in X$ will be the endpoint of the lift $\widetilde{\gamma}$ of some path $\gamma $ in $Z$ based at $z$---if $x_i=\Psi (x_i)$ is the initial point of this lift, then we must have $\Psi \circ \widetilde{\gamma } = \widetilde{\gamma },$ so that in particular $\Psi (x) = \Psi \circ \widetilde{\gamma } (1) = \widetilde{\gamma } (1) = x.$
This gives Property 2).
The proof of Property 3) is very similar, and may also be found, for instance, in~\cite[Proposition 1.3]{Cukierman}.\\\\
It remains to show Property 4).
We do so by first constructing a map $\rat{\Psi_\sigma}{X}{X}$ pointwise via lifting paths.
The argument is analagous to the proof in~\cite[Propsition 1.39]{Hatcher}.
Fix $x_0\in X_z.$ For generic $x\in X,$ there exists a path $\alpha_x : [0,1] \to X$ from $x_0$ to $x$ whose image in $Z$ is contained in a regular locus for $\ratnoname{X}{Z}.$
Define $\overline{\alpha_x}$ to be the lift based at $\sigma (x_0)$ whose image in $Z$ coincides with the image of $\alpha .$
We define 
\begin{center}
$\rat{\Psi_\sigma }{X}{X}$\\
$x \mapsto \overline{\alpha_x} (1).$
\end{center}
\begin{figure}
\begin{center}
\includegraphics[width=18em]{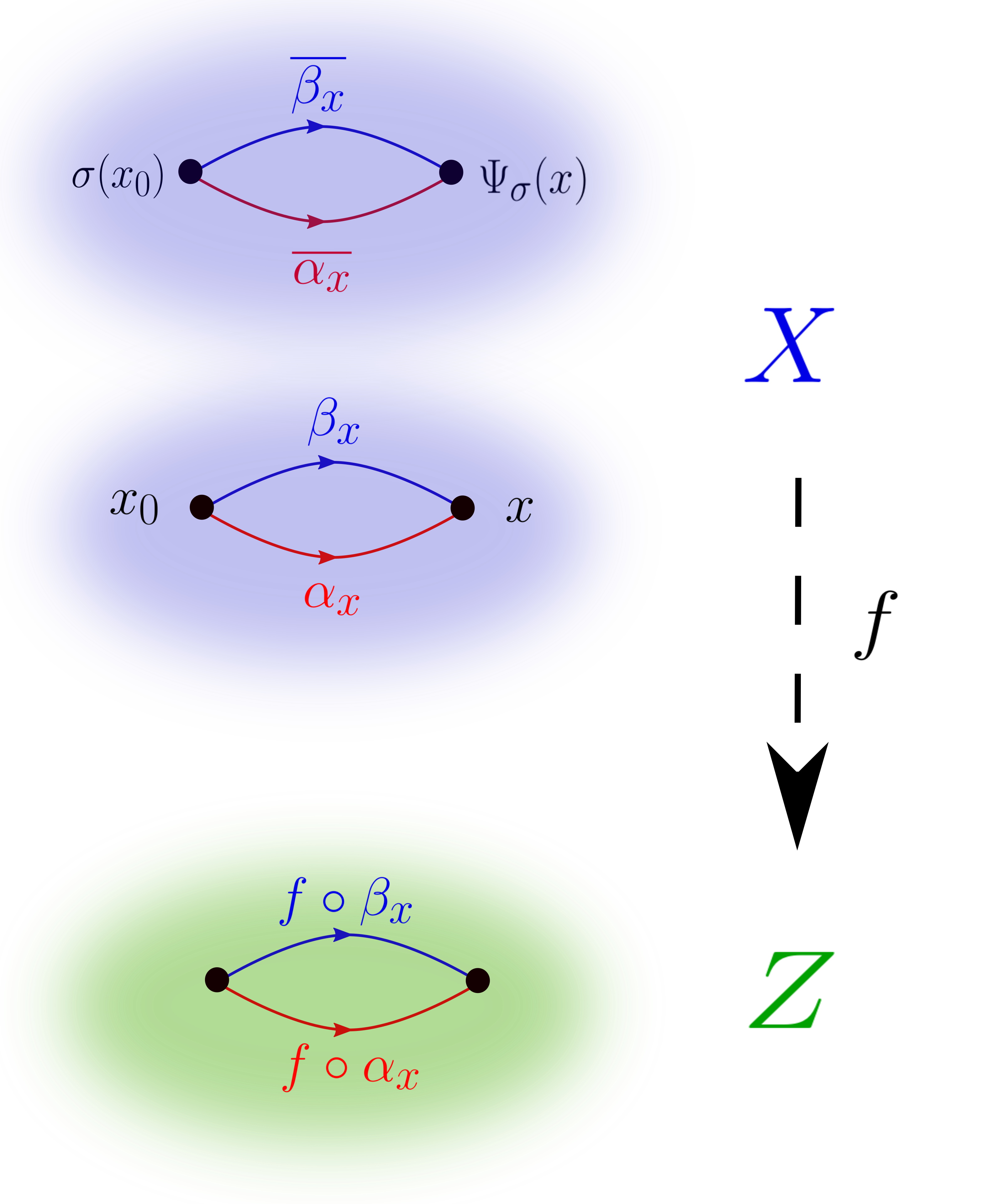}
\end{center}
\caption{Construction and well-definedness of $\Psi_\sigma .$}\label{fig:deck}
\end{figure}
First of all, we must show that $\Psi_\sigma$ is well-defined.
This means that for any other path $\beta_x$ from $x_0$ to $x$ we must have $\overline{\beta_x}(1) = \overline{\alpha_x}(1).$
We refer the reader to Figure~\ref{fig:deck} to more easily follow the argument.
Consider the loop $\gamma $ based at $f(x_0)$ in $Z$ obtained by concatenating $\overleftarrow{f\circ \beta_x}$ (the reverse of the path $f\circ \beta_x$) with $f\circ \alpha_x .$
Since $\sigma $ and $\sigma_\gamma $ commute, we have that
\begin{align*}
\sigma_\gamma (\sigma (x_0)) &= \sigma (\sigma_\gamma (x_0))\\
&= \sigma (x_0).
\end{align*}
Thus $\sigma_\gamma $ fixes $\sigma (x_0),$
and it follows that the lift $\widetilde{\gamma }$ based at $\sigma (x_0)$ is a loop.
This implies that $\overline{\alpha_x}$ and $\overline{\beta_x}$ have the same terminal point $\Psi_\sigma (x),$ proving well-definedness.\\\\
Consider now an arbitrary $x\in X_z.$ 
Note that in this case $f\circ \alpha_x $ is a loop.
We calculate
\begin{align*}
\sigma (x) &= \sigma ( \sigma_{f\circ \alpha_x} (x_0))\\
&= \sigma_{f\circ \alpha_x} (\sigma (x_0))\\
&= \Psi_\sigma (x).
\end{align*}
Thus, restricting $\Psi_\sigma $ to $X_z$ yields the permutation $\sigma .$
Moreover, by definition of $\Psi_\sigma$ we have that $f \circ \Psi_\sigma = f$ on the locus of points where both maps are defined.
It remains to show that $\Psi_\sigma $ is a rational map, since then it will also follow that $\Psi_{\sigma^{-1}}$ is a rational inverse.
First we note that, in a suitably small neighborhood of any generic point $x\in X$, we can write $\Psi_\sigma = g_x \circ f,$ where $g_x$ is a holomorphic local inverse of $f.$
Such an expression for $\Psi_\sigma$ exists for any $x$ in some Zariski open subset of $X.$
It follows that $\Psi_\sigma $ is a meromorphic map from $X$ to itself---in other words, it is holomorphic after restricting to a Zariski-open $U\subset X.$
To finish the proof, we may use the well-known fact that all meromorphic
maps between projective varieties are rational.\footnote{We recall the words of Mumford~\cite[Ch.~4]{Mumford}: ``This should be viewed as a generalization of the old result that the only everywhere meromorphic functions on $\CC \cup \{ \infty \} $ are rational functions."}
Indeed, if $X$ and $Z$ are quasiprojective varieties, we may replace them with their projective closures $\overline{X}, \overline{Z},$ to get a birationally equivalent $\ratnoname{\overline{X}}{\overline{Z}}.$
\end{proof}

\begin{prop}
\label{prop:deckImpliesDecomposable}
A branched cover $\ratnoname{X}{Z}$ of degree $d$ with a nontrivial deck transformation $\Psi $ is either decomposable or its Galois/monodromy group is cyclic of order $d.$
In the latter case, $\mon (X/Z)$ is imprimitive precisely when $d$ is composite.
\end{prop}

\begin{proof}
Partition $X_z$ into the orbits under repeated application of $\Psi .$
This partition is preserved under the monodromy action. Thus, the Galois group is imprimitive if this partition is nontrivial.
Otherwise, the action of $\Psi $ on $X_z$ generates a cyclic group $C_d \subset S_d.$
Letting $\cent (\cdot )$ denote the centralizer in $S_d$, we have $C_d \subset \cent (\mon (X/Z)),$ which holds and only if $\mon (X/Z) \subset \cent (C_d) = C_d .$
Since $\mon (X/Z)$ is transitive, we must have $\mon (X/Z) = C_d.$
\end{proof}

In general, a decomposable branched cover need not have any deck transformations.
However, a converse to Proposition~\ref{prop:deckImpliesDecomposable} does hold in a special case frequently encountered in practice.

\begin{proposition} \label{prop:bir_order_2}
$\ratnoname{X}{Z}$ has a deck transformation of order $2$ if and only if $\mon (X/Z)$ has a block of size $2$.
\end{proposition}
\begin{proof} $\Rightarrow $ As in Proposition~\ref{prop:deckImpliesDecomposable}. $\Leftarrow$ Proposition~\ref{prop:decomposable_iff_imprimitive} gives a factorization such that $\CC (X) / \CC (Y)$ is a degree $2$ extension, and thus always Galois.
\end{proof}

\subsection{Numerically computing Galois/monodromy groups}\label{subsec:num_methods}
In this paper, our main interest is in minimal problems.
These typically give rise to branched covers of the form $\ratnoname{X}{\CC^m},$ at least up to birational equivalence.
We again emphasize that being a branched cover implies $\dim X = m$ and, for \emph{generic measurements} $z\in \CC^m,$ that the fiber $X_z$ is finite.
This is essentially the definition of a minimal problem used in~\cite{PLMP,PL1P}, where a problem is minimal if and only if its joint camera map is a branched cover.
In practice it is usually enough to consider the case where $X $ is a subvariety of $\CC^n \times \CC^m,$ and the map $X \to \CC^m$ is coordinate projection.
In this case, the ideal $\mathcal{I}_X$ of all polynomials in $\CC[x_1, \ldots , x_n, z_1, \ldots , z_m]$ that vanish on $X$ is prime.\\\\
However, one advantage of computing Galois/monodromy groups numerically is that we do not need to know all generators of $\mathcal{I}_X.$
All that is really needed are
\begin{itemize}
    \item[1)] the ability to sample a generic point $(x^*, z^*)\in X,$ and
    \item[2)] a set of $n$ equations vanishing on $X,$ \begin{equation}\label{eq:witnessSystem}
F(x;z) = F(x_1,\dots,x_n;z_1,\dots,z_m) = \left[\begin{array}{c} f_1(x_1,\dots,x_n;z_1,\dots,z_m) \\
\vdots \\ f_n(x_1,\dots,x_n;z_1,\dots,z_m)
\end{array}\right],
\end{equation}
such that the Jacobian $d_x \, F(x^*;z^*)$ is an invertible $n\times n$ matrix for generic $(x^*, z^*)\in X.$
We say that equations~\eqref{eq:witnessSystem} form a \deff{well-constrained system} for the branched cover $X \to \CC^m.$
\end{itemize}
In equation~\eqref{eq:witnessSystem}, $x_1,\ldots , x_n$ are usually called the \deff{variables} and $z_1, \ldots , z_m$ are usually called the \deff{parameters}.
For generic $(x^*, z^*) \in X$ and generic $z\in \CC^m,$ the path $\alpha : [0,1] \to \CC^m$ defined by the straight-line segment $\alpha (t) = (1-t)\cdot z^* + t\cdot z$ will have a lift $\widetilde{\alpha } (t)$ to $X$ based at $\widetilde{\alpha} (0) = (x^*, z^*).$ 
By genericity, the lifted path $\widetilde{\alpha}:[0,1] \to X$ will not intersect the subvariety of $X$ where $d_x \, F$ is singular for any $t\in [0,1]$---see~\cite[Lemma 7.12]{SW}.
This implies that $\widetilde{\alpha}$ can be numerically approximated by applying numerical continuation methods to the \deff{parameter homotopy} 
\begin{equation}\label{eq:ParameterHomotopy}
H(x,t) = F \left(x;  \alpha (t) \right) = 0,
\end{equation}
which connects known solutions of the \deff{start system} $H(x,0) = f(x;z^*) = 0$ to solutions of the \deff{target system}\footnote{This convention is chosen to agree with the notation of the previous section. 
We note that this is the opposite of the common convention of placing start systems at $t=1$ and target systems at $t=0$ which, although mathematically equivalent, is more natural from the numerical point of view.}  $H(x,1) = F(x;z) = 0.$
The solution curves $x(t)$ satisfying $H(x(t), t)=0$ and the initial condition $x(0) = x^*$ will stay on our irreducible variety $X$ with probability-one, and hence $\widetilde{\alpha } (t) = (x(t), \alpha (t)).$
Thus, although the variety $X$ need not be a complete intersection, appropriate use of a well-constrained system enables us to compute solutions on $X$ using the same number of equations as unknowns.
This fits into an established paradigm of numerical algebraic geometry where overdetermined parameterized polynomial systems may be solved by reduction to a well-contrained system (see also~\cite[Sec 6.4]{BHSW13},~\cite{HauensteinRegan}.)
In practice, our numerical approximations to $\widetilde{\alpha}(t)$ remain ``close'' to $X$ with some probability that depends on the conditioning and the implementation of the numerical methods.\\\\
By numerically continuing solutions along some path $\beta (t)$ from $z$ to $z^*$,
and then along some other path $\alpha (t)$ from
$z^*$ to $z$, the concatenated path $\gamma = \beta * \alpha $ is a loop based at $z$ which induces a monodromy permutation $\sigma_{\gamma } \in \mon (X/\CC^m ; z).$
This simple observation motivates numerous applications of monodromy in numerical algebraic geometry: computing the fibers $X_z$ (addressed in~\cite{MSpaper,MARTINDELCAMPO2017559}) computing the Galois/monodromy group (addressed in~\cite{NumGalois,GaloisSchubert}), and in additional applications ranging from numerical irreducible decomposition~\cite{NIDpaper} to kinematics~\cite{RealMonodromy}.
To compute $\mon (X/\CC^m)$ in this paper, we simply generate some number of loops $\gamma_1, \ldots , \gamma_k$ in $\CC^m,$ with $k$ ranging from $4$ (usually sufficient when $\mon (X/\CC^m)$ is full-symmetric) to as large as $50.$
This adds additional uncertainty to our numerical computations, since \emph{a priori} we only know that $\langle \sigma_{\gamma_1} , \ldots , \sigma_{\gamma_k} \rangle $ is a subgroup of $\mon (X/\CC^m; z).$
In principle, this additional uncertainty could be avoided by a more computation-heavy approach like the branch-point method in~\cite{NumGalois}.
Nevertheless, we feel reasonably confident in the Galois/monodromy group computations which we have report, which have been validated through repeated runs and analyzing decompositions in several of the imprimitive cases.

\section{Absolute pose problems}\label{sec:abspose}
In this section, we apply the mathematical framework of the previous section to absolute pose problems involving combinations of point/line features appearing in work of Ramalingam et al.~\cite{ramalingam}
Absolute camera pose estimation is one of the main problems of computer vision~\cite{Ameller02camerapose,RANSAC,DBLP:journals/ijcv/HaralickLON94,DBLP:journals/ijcv/LepetitMF09,DBLP:journals/pami/QuanL99,DBLP:conf/issac/ReidTZ03,DBLP:conf/iccv/Triggs99,DBLP:journals/jmiv/WuH06}.
Although the problems considered here are of low degree, computing the Galois/monodromy groups yields new insights which might be applied to building better solvers for these problems.

We begin formulating these problems in the language of branched covers.
Our general task is to determine a calibrated camera matrix $\cam{\R}{\tran}$ from correspondence data between the scene and images.
We let $p$ and $l$ be the numbers of point-point and line-line correspondences, respectively, between 3D and 2D. 
The total space of our branched cover is
\[
X_{p,l} = \left( \PP^3 \right)^p \times \left( \GG_{1,3} \right)^l 
\times \SECC (3)
\]
where $\GG_{1,3}$ denotes the Grassmannian of lines in $\PP^3.$
The base space equals
\[
Z_{p, l}  = \left( \PP^3 \right)^p \times \left( \GG_{1,3} \right)^l \times \left( \PP^2 \right)^p \times \left( \GG_{1,2} \right)^l,
\]
where $\GG_{1,2}$ denotes the Grassmannian of lines in $\PP^2$ and $\rat{f_{p,l}}{X_{p,l}}{Z_{p.l}}$ is the map that ``takes pictures": 
\begin{align*}
\left(
X_1, \, \ldots , \, X_p, \, \overline{L_1 \, L_1'}, \, \ldots , \, \overline{L_l \, L_l '}, \, \cam{\R}{\tran} \right) \mapsto \\
\bigg(
X_1, \, \ldots , \, X_p, \, \overline{L_1 \, L_1'}, \, \ldots , \, \overline{L_l \, L_l '}, \, 
\cam{\R}{\tran}\, X_1, \, \ldots , \, \cam{\R}{\tran}\, X_p, \\
\overline{\cam{\R}{\tran}\, L_1 \, \, \cam{\R}{\tran}\, L_1'}, \, \ldots , \, \overline{\cam{\R}{\tran}\,  L_l \, \, \cam{\R}{\tran}\, L_l '}
\bigg)
\end{align*}
(here $\overline{L \, L'}$ is the line spanned by $L$ and $L'.$)
Counting dimensions gives $\dim X_{p,l} = 3 (p+l) + 6$ and $\dim Z_{p,l} = 5 (p+l).$
Equating the two, we see that the only possibilities are $(p,l)=(3,0),$ $(2,1),$ $(1,2), $   $(0,3).$
The first case corresponds to the P3P problem.

\begin{result}
\label{res:abs-pose}
The full list of Galois/monodromy groups of branched covers $f_{p,l}$ is as follows:
\begin{align*}
\mon (X_{3,0}/Z_{3,0}) &\cong S_2 \wr S_4 \cap A_8 \hookrightarrow S_8\\
\mon (X_{2,1}/Z_{2,1}) &\cong S_2 \wr S_2 \cap A_4 \cong C_2 \times C_2 \hookrightarrow S_4\\
\mon (X_{1,2}/Z_{1,2}) &\cong S_2 \wr S_4 \cap A_8 \hookrightarrow S_8\\
\mon (X_{0,3}/Z_{0,3}) &\cong S_8.
\end{align*}
\end{result}
We interpret Result~\ref{res:abs-pose} in two separate subsections, corresponding to the ``unmixed cases" $(p,l) \in \{ (3,0), \, (0,3) \}$ and the more interesting ``mixed cases" $(p,l) \in \{ (2,1), \, (1,2) \}.$
We note the respective degrees $8, \, 4, \, 8, \, 8$ agree with those reported in~\cite{ramalingam}, in which these problems were formulated using different systems of equations.
The systems of equations defining the parameter homotopies used for Result~\ref{res:abs-pose} were constructed as follows:
\begin{itemize}
\item Points in the world are represented by $4\times 1$ matrices $X_1, \ldots , X_p.$
\item Points in the image are represented by $3\times 1$ matrices $x_1, \ldots , x_p.$
\item Lines in the world are represented as kernels of $2\times 4 $ matrices $[\N_1  \mid \N_1 ']^\top, \ldots , [\N_l \mid  \N_l']^\top.$
\item Lines in the image are represented as kernels of $1\times 3$ matrices $\nn_1^\top , \ldots \nn_l^\top.$
\item We enforce rank constraints by the vanishing of maximal minors of certain matrices:
\begin{itemize}
    \item point-to-point: $\rank \Big( \cam{\R}{\tran} \, X_i \mid x_i \Big) \le 1$ for $i=1,\ldots , p$
    \item line-to-line: $\rank \Big( \N_i \mid \N_i ' \mid \cam{\R}{\tran}^\top \nn_i \Big) \le 2$ for $i=1,\ldots , l$
\end{itemize}
\item Any square subsystem of these maximal minors whose Jacobian has full rank gives a well-constrained system in the sense of Section~\ref{subsec:num_methods} (the variables being $\cam{\R}{\tran},$ and the parameters being all points and lines.)
Although the square subsystem may have excess solutions, they are not in the same orbit as the geometrically relevant solutions under the monodromy action.
Alternatively, we can get a well-constrained system by randomization as in~\cite[Sec 6.4]{BHSW13}.
The monodromy group does not depend on the choice of well-constrained system.
\end{itemize}

\subsection{The unmixed cases: \texorpdfstring{$(p,l)=(3,0), \, (0, 3)$}{(p,l)=(3,0),  (0, 3)}}

The case $(p,l) = (3,0)$ reduces to solving the P3P problem as formulated in Equations~\eqref{eq:p3p}.
The literature on this problem is vast, and the earliest work~\cite{Grunert-1841} pre-dates the field of computer vision by more than a century.
The degree of this problem is $8$ and the Galois/monodromy group is a subgroup of $S_2 \wr S_4$ due to the sign symmetry.
In the terminology of Brysiewicz et al.~\cite{Brysiewicz}, Equations~\eqref{eq:p3p} are a lacunary polynomial system whose monomial supports span a proper sublattice of $\ZZ^3$ with finite index.
In the setting of that paper, we would consider the family of all systems with the same monomial supports as in~\eqref{eq:p3p}
\begin{equation}
\label{eq:fullsupport}
\begin{split}
h_{1,2} = A x_1^2 + B x_2^2 + C x_1 x_2 + D\\
h_{1,3} = E x_1^2 + F x_3^2 + G x_1 x_3 + H\\
h_{2,3} = I x_2^2 + J x_3^2 + K x_2 x_3 + L 
\end{split}
\end{equation}
This gives a branched cover $X_h \to \CC^{12}$ where $X_h = V(h_{1,2}, h_{1,3}, h_{2,3}) \subset \CC^3 \times \CC^{12}.$
On the other hand, for P3P the natural branched cover is $X_f \to \CC^6,$ where $X_f\subset \CC^{3} \times \CC^{6}.$
We find numerically that $\mon (X_h / \CC^{12})$ is the full wreath product $S_2 \wr S_4,$ whereas our numerical experiments suggest that the Galois/monodromy group for P3P is the \emph{proper subgroup} $S_2 \wr S_4 \cap A_8.$ 

To certify the result of our numerical monodromy computation, we can compute the Galois group for P3P using symbolic computation.
Consider $I = \langle f_{1,2}, f_{1,3}, f_{2,3} \rangle $ as an ideal in a polynomial ring $\mathbb{F} [x_1,x_2,x_3]$ whose coefficient field is $\mathbb{F} = \CC (Z) = \CC (\vec{c}, \vec{d}).$  
The dimension and degree of $I$ are $0$ and $8.$
We can compute a lexicographic Gr\"{o}bner basis for $I$ with $x_1 > x_2 > x_3$ in a matter of seconds using the FGLM algorithm~\cite{faugere1993efficient}, implemented for Macaulay2~\cite{M2} in the package \texttt{FGLM}~\cite{FGLMSource}.
The Gr\"{o}bner basis $G = \{ g_1, g_2, g_3 \}$ has the form predicted by the Shape lemma:
\begin{align*}
g_1 (x_1,x_2,x_3) &= x_1 + r_1 (\vec{c}, \vec{d} ) \, x_3 \\
g_2 (x_2,x_3) &= x_2 + r_2(\vec{c}, \vec{d} ) \, x_3\\
g_3 (x_3) &= x_3^8 + A (\vec{c}, \vec{d} ) \, x_3^6 + B (\vec{c}, \vec{d} ) \, x_3^4 + C (\vec{c}, \vec{d} ) \, x_3^2 + D (\vec{c}, \vec{d} ), 
\end{align*}
for particular rational functions $r_1,r_2,A,B,C,D \in \mathbb{F}.$
We see that $x_3$ is a primitive element for the extension $\CC (X) / \CC (Z).$
To verify that $\mon (X/Z) \cong \gal (X/Z)$ is contained in $A_8,$ it suffices to show that the discriminant of $g_3$ is square.
For arbitrary coefficients $(A,B,C,D),$ the discriminant of $x_3^8 + A  \, x_3^6 + B  \, x_3^4 + C \, x_3^2 + D$ is the product of $D $ and a square. 
For P3P, $D(\vec{c}, \vec{d} )$ is also a square.

Factorizations of P3P are also classical: from Equations~\eqref{eq:p3p}, we have
\begin{equation}
\label{eq:p3p-factor}
\begin{split}
y_1 (1+ y_2^2 - c_{1,2} y_2 ) - d_{1,2}^2=0\\
y_1 (1+ y_3^2 - c_{1,3} y_3 ) - d_{1,3}^2=0\\
y_1 (y_2^2 + y_3^2 - c_{2,3} y_2 y_3) - d_{2,3}^2 = 0
\end{split}
\end{equation}
where $y_1 = x_1^2, \, y_2 = x_2 / x_1, \, y_3 = x_3/x_1$ are separating invariants~\cite{Kemper} for the action of the deck transformation group.
We see that even when $X \to Z$ is regular in the definition of a factorization~\eqref{eq:factorization-diagram}, the maps $\ratnoname{X}{Y}$ and $\ratnoname{Y}{Z}$ need not be. 

On the other hand, the branched cover $\ratnoname{X_{0,3}}{Z_{0,3}}$ for absolute pose from $3$ lines is indecomposable, since its Galois/monodromy group $S_8$ acts primitively on the set of $8$ solutions.
Thus, any algebraic algorithm for solving this problem must be capable of computing the roots of a polynomial of degree $8$ or higher.

\subsection{The mixed cases: \texorpdfstring{$(p,l)=(2,1), \, (1, 2)$}{(p,l)=(2,1), (1,2)}}

Proposition~\ref{prop:deck-centralizer}
shows that each of the mixed cases has a nontrivial deck transformation group:
we have that $\Aut (X_{2,1} / Z_{2,1} ) \cong C_2 \times C_2$ and $\Aut (X_{1,2} / Z_{1,2} ) \cong C_2.$
Using the rank constraints described above, we were able to observe numerically that solutions in the same block for both of these mixed cases differed by a reflection. 
These deck transformations take on a particularly simple form after changing coordinates as in~\cite{ramalingam}.

For the case $(p,l) = (2,1),$ the formulation~\cite[Equations 4,5]{ramalingam} makes use of a clever choice of reference frames to get equations
\begin{equation}
\label{eq:ramalingam-alt}
\begin{split}
A \, \bs{X} - b = 0 \\
R_{1,1}^2 + R_{2,1}^2 + R_{3,1}^2 - 1 = 0\\
R_{2,1}^2 + R_{2,2}^2 + R_{2,3}^2 - 1 = 0
\end{split}    
\end{equation}
where $A$ and $b$ are $6\times 8$ and $8\times 1$ matrices depending on the given data, and 
\[
\bs{X} = [R_{1,1}, R_{2,1}, R_{3,1}, R_{2,2}, R_{2,3}, t_1, t_2, t_3]^\top
\]
is a vector of indeterminates.
Using \texttt{FGLM} as in the previous section, we discover new constraints
\begin{align*}
R_{3,1}^2 + \psi_1 (A,b) = 0 \\
t_3^2 + 2 t_3 + \psi_2 (A,b) = 0, 
\end{align*}
for particular rational functions $\psi_1, \psi_2$ in the data,
which did not appear in~\eqref{eq:ramalingam-alt} originally.
Formulas for the deck transformations of this Galois cover follow by way of the basic Example~\ref{ex:deck-quad}.
The remaining constraints output by \texttt{FGLM} are, as expected, of the form
\begin{align*}
R_{i,j} + \ell_{i,j} (R_{3,1})=0\\
t_j + \ell_j (t_3) =0
\end{align*}
for linear forms $\ell_j, \ell_{i,j}$ over the coefficient field $\QQ (A,b).$
For this very special problem, the Gr\"{o}bner basis elements are surprisingly compact. 
This suggests, as an alternative to the solution proposed in~\cite{ramalingam}, that we may solve for the rotation and translation independently. 

Likewise, for the $(p,l) = (1,2)$ case, using the similar formulation of~\cite[Equations 7,8]{ramalingam}, we discover the following symmetry in the solutions ($\ee_3\in \RR^3$ is the third standard basis vector):
\[
(\vec{R}, \, \tran) \mapsto (-\vec{R}, -\tran - 2 \ee_3). 
\]
We note that in this formulation, $\vec{R}$ contains only the first two rows of the unknown rotation matrix.
In hindsight, this symmetry is quite easy to verify.
However, we stress that computing the Galois/monodromy group was what led us to discover it.

\section{Relative pose problems}
\label{sec:rel-pose}
Our interest in Galois/monodromy groups of minimal problems began when we computed, using Bertini~\cite{Bertini}, that the Galois/monodromy group of the five-point problem was $S_2 \wr S_{10} \cap A_{20}.$
Much like P3P, we had initially expected the full wreath product.
Since then, we have computed Galois/monodromy groups of many minimal problems using Bertini and the Macaulay2 package \texttt{MonodromySolver}~\cite{MonodromySolver}.
Overall, we agree with the assessment of Esterov and Lang that Galois/monodromy groups of structured polynomial systems are \emph{``unexpectedly rich}''~\cite{esterov2018sparse}.

Many of the problems we considered appeared previously in~\cite[Table 1]{PLMP}.
In that work, branched covers were represented by pictograms. 
For instance, the five-point problem was denoted by $\fivepointproblem .$
The problems $\cleveland$ and $\chicago$ were introduced in~\cite{Joe,Ricardo}, respectively.
Our computations that $\mon \left(\cleveland \right) \cong S_{216}$ and $\mon \left(\chicago \right)\cong S_{312}$ show that the homotopy solvers developed in~\cite{Ricardo} are \emph{optimal} in the sense of tracking the fewest paths possible. 
The problem $\partialminimal$ appearing in~\cite{PL1P} can be thought of as P3P fibered over the five-point problem.
Unlike the majority of problems studied here,
this composite minimal problem has an intermediate field $\CC (Z) \subsetneq \CC (Y) \subsetneq \CC (X)$ which is not the fixed field of some subgroup of $\Aut (X/Z).$

\begin{result}\label{res:plmp}
Among all minimal problems of degree $< 1, 000$ appearing in~\cite[Table 1]{PLMP}, all have either an imprimitive or full symmetric Galois/monodromy group.
The imprimitive cases are:
\begin{align*}
\mon \left(\twoviewhomography \right) &\cong (C_2)^2 \rtimes (S_2 \wr S_3 \cap A_6) \hookrightarrow S_{12} \\
\mon \left( \fivepointcollinear \right) &\cong S_2 \wr S_{8} \cap A_{16} \hookrightarrow S_{16}\\
\mon \left(\fivepointproblem\right) &\cong S_2 \wr S_{10} \cap A_{20} \hookrightarrow S_{20}\\
\mon \left( \threeviewhomography \right) &\cong S_2 \wr \left( S_2 \wr S_{16} \cap A_{32} \right) \cap A_{64} \hookrightarrow S_{64}\\
\mon \left( \hedgehog \right) &\cong \left(C_2\right)^4  \rtimes \left( \left(C_2\right)^4 \rtimes \left( S_2 \wr (S_2 \wr S_4) \right) \right) \hookrightarrow S_{64}\\
\mon \left( \fourcam \right) &\cong \left(C_2\right)^2 \rtimes \left(
C_2^2 \rtimes \left(
S_2 \wr \left(
S_2 \wr S_2 \cap A_4
\right)
\cap A_8
\right)
\right)
\hookrightarrow S_{32}. 
\end{align*}
\end{result}
For the sake of uniformity, we have used the semidirect product $\rtimes $ to indicate subgroups of an appropriate wreath product.
Thus, for instance, for $\mon \left( \fourcam \right),$ the outermost $(C_2)^2 $ should be regarded as a subgroup of $(S_2)^{16},$ and the innermost as a subgroup of $(S_2)^{8}.$
Much to our surprise, the group $\mon \left( \fourcam \right)$ turns out to be solvable.

Clearly there are many minimal problems waiting to be decomposed.
With increasingly efficient and user-friendly homotopy continuation software like Bertini, \texttt{HomotopyContinuation.jl}~\cite{HCJL}, and various numerical packages for Macaulay2 (summarized in~\cite{leykin2018homotopy}), we see no obstacles to computing even more Galois/monodromy groups of interest to computer vision in the future.

In the remainder of this section, we give particular attention to three problems appearing in Result~\ref{res:plmp}.
The first is the five-point problem $\fivepointproblem .$ 
The second is the five-point problem for data which lie in a ``V"-shape $\twoviewhomography .$
Decomposing the associated branched cover reveals the classical planar calibrated homography problem.
The third, denoted $\threeviewhomography,$ is a special case of the notorious four-points-in-three-views problem~\cite{NisterSchaf} in which three of the corresponding points are collinear.

\subsection{Five-point relative pose}
\label{subsec:5pp}
We now return to the five-point problem, retaining the notation from Examples~\ref{ex:5pp} and~\ref{ex:resolve-twisted}.
It is well-known fact that the branched cover $X\to Z$ is decomposable.
The intermediate variety is
\[
Y = \{ 
\left( \E , \, \left (\xx_1, \, \ldots , \, \yy_5 \right) \right) \in \Ess \times Z \mid \yy_i^\top \E \, \xx_i = 0 ,\, i=1,\, \ldots , \, 5 
\} .
\]
The equations defining $Y$ give the formulation of the five-point problem studied in several seminal works~\cite{Demazure,Kruppa,FaugerasMaybank}, and used in state of the art five-point solvers such as Nist\'{e}r's~\cite{Nister}.
The fact that $Y\to Z$ is a generically $10$-$1$ map is closely related to the fact that $\Ess$ is a projective variety of dimension $5$ and degree $10.$
Although the individual linear equations $\yy_i^\top \E \, \xx_i =0$ are special, considered together they determine a dominant rational map $\ratnoname{Z}{\GG_{3,8}}.$ 
This fact follows from the trisecant lemma of classical algebraic geometry (see~\cite[p.~134]{Mumford} for one statement of this result, or~\cite[Sec.~5.2.3]{Maybank} for an alternative explanation of this fact.)

In formulating the five-point problem, it is possible to normalize the translations and depths in various ways.
For instance, we might use the normalization $\norm{\tran}^2 = 1,$ giving rise to a branched cover $W \to Z$ with $W \subset \SOCC (3) \times \CC^{13} \times Z .$
This branched cover decomposes as
\[
W \dashrightarrow X \rightarrow Y \rightarrow Z.
\]
The four elements in the generic fiber of $W\to Y$ are illustrated in~\cite[Figure 9.12]{HZ-2003}.
The various Galois/monodromy groups for the five-point problem are summarized in Result~\ref{res:5pp}.
The result $\mon (Y/Z) \cong S_{10}$ gives numerical confirmation of the results of~\cite{NisterHartley}, which imply that $\Gal \left( \galclo{\QQ (Y)} / \QQ (Z) \right)\cong S_{10}$ by appeal to Hilbert's irreducibility theorem~\cite[Proposition 3.3.5]{Serre}. 
\begin{result}
\label{res:5pp}
The Galois/monodromy groups of the five-point problem are as follows:
\begin{align*}
\mon (W/Z) &\cong \left(C_2\right)^9 \rtimes (S_2 \wr S_{10}) \hookrightarrow S_{40} \\ 
\mon (X/Z) &\cong S_2 \wr S_{10} \cap A_{20} \hookrightarrow S_{20}\\ 
\mon (Y/Z) &\cong S_{10},
\end{align*}
where $C_2^9$ is the subgroup of  $(\tau_1, \ldots , \tau_{20}) \in (S_2)^{20}$ generated by $\tau_{2i-1} \tau_{2i}$ for $i=1,\ldots , 20.$
\end{result}

\subsection{Two-view homography}
\label{subsec:2viewhomography}
We now move on to a minimal problem for planar scenes.
We use similar notation as in the five-point problem, with constraints
\begin{equation} \label{eq:4p2v}
    \begin{split}
        \R^\top\R = \I, \quad \det \R = 1,
        \\
        \beta_i\yy_i = \R\alpha_i\xx_i + \mathbf{t}, \;\; \alpha_i, \beta_i \neq 0, \quad i = 1,\dots,4,
        \\
        \det\left(\begin{bmatrix} \alpha_1\xx_1 & \alpha_2\xx_2 & \alpha_3\xx_3 & \alpha_4\xx_4 \\ 1 & 1 & 1 & 1 \end{bmatrix}\right) = 0.
    \end{split}
\end{equation}
Our incidence variety $X$ is given by 
\[ X \subseteq \SOCC (3) \times \PP_\CC^{10} \times \underbrace{\left(\CC^2 \times \{1 \}\right)^4 \times \left(\CC^2 \times \{1 \}\right)^4}_{Z} \]
such that Equations~\ref{eq:4p2v} hold for all $\left(\R, (\mathbf{t}, \alpha_1, \ldots , \alpha_4, \beta_1, \ldots , \beta_4), (\xx_1,\ldots , \xx_4, \yy_1 , \ldots , \yy_4) \right) \in X.$
Projection of $X$ onto $Z$ defines a branched cover of degree $12.$ 
This branched cover is birationally equivalent to the joint camera map of the problem $\twoviewhomography$ from~\cite{PLMP}, since the fifth point on both lines in each image is generically determined from the other four points.
Result~\ref{res:plmp} tells us the Galois/monodromy group is $C_2 \times C_2 \rtimes (S_2 \wr S_3 \cap A_6).$
The GAP command \texttt{MinimalGeneratingSet} shows that this group is minimally generated by two permutations: in cycle notation,
\begin{equation}
\label{eq:homography-permutations}
\mon (X/Z) \cong \Big\langle
(1 \,\,\, 2) (3 \,\,\, 4) (5 \,\,\, 12 \,\,\, 8 \,\,\, 9) (6 \,\,\, 11 \,\,\, 7 \,\,\, 10),\,
(1 \,\,\, 11 \,\,\, 5) (2 \,\,\, 10 \,\,\, 8) (3 \,\,\, 9 \,\,\, 7) (4 \,\,\, 12 \,\,\, 6)
\Big\rangle. \end{equation}

\begin{figure}[H]
\centering
  \begin{tikzpicture}
  \node (1) at (0,0) {$[12]$};
  \node (2) at (0,-2) {$[4]$};
  \node (3) at (-2,-4) {$[2]$};
  \node (4) at (0,-4) {$\{ 1, 3 \}$};
  \node (5) at (2,-4) {$\{ 1, 4 \}$};
  \node (6) at (0,-6) {$[1]$};
  
  \path[-]
  (2) edge node[left] {$3$} (1)
  (3) edge node[above left] {$2$} (2)
  (4) edge node[left] {$2$} (2)
  (5) edge node[above right] {$2$} (2)
  (6) edge node[below left] {$2$} (3)
  (6) edge node[left] {$2$} (4)
  (6) edge node[below right] {$2$} (5);
  \end{tikzpicture}
  \phantom{ffffffffffffffff}
  \begin{tikzpicture}
  \node (1) at (0,0) {$\CC(Z)$};
  \node (2) at (0,-2) {$\CC(X)_{\langle\Psi_1,\Psi_2\rangle}$};
  \node (3) at (-2,-4) {$\CC(X)_{\langle\Psi_1\circ\Psi_2\rangle}$};
  \node (4) at (0,-4) {$\CC(X)_{\langle\Psi_1\rangle}$};
  \node (5) at (2,-4) {$\CC(X)_{\langle\Psi_2\rangle}$};
  \node (6) at (0,-6) {$\CC(X)$};
  
  \path[-]
  (1) edge node[left] {$3$} (2)
  (2) edge node[above left] {$2$} (3)
  (2) edge node[left] {$2$} (4)
  (2) edge node[above right] {$2$} (5)
  (3) edge node[below left] {$2$} (6)
  (4) edge node[left] {$2$} (6)
  (5) edge node[below right] {$2$} (6);
\end{tikzpicture}
\caption{Correspondence between block systems (left) and intermediate fields (right) for the calibrated homography problem.
The notation $K_H$ means the intermediate field of an extension $K/F$ fixed elementwise by a subgroup $H \le \Aut(K/F).$
}
\label{fig:blocks_fields}
\end{figure}
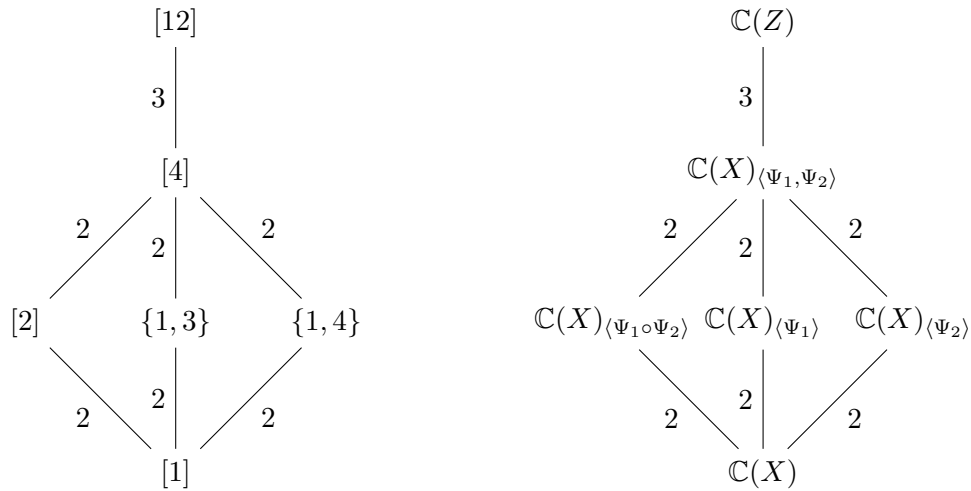

The lattice of block systems is depicted on the left in Figure~\ref{fig:blocks_fields}.
The vertex labels correspond to stabilizer subgroups of $\mon (X/Z)$, and the edges are labeled by the degrees of maps appearing in some decomposition of the form in Equation~\eqref{eq:decomposition}.
To the right is the inverted lattice of intermediate fields.
Like the majority of examples in this paper, $\CC (X) / \CC (Z)$ is not a Galois extension.

Before we determine a decomposition, we first describe the group of deck transformations.
The centralizer in $S_{12}$ is
\[ 
\Big\langle 
(1 \,\,\, 3) (2 \,\,\, 4) (5 \,\,\, 7) (6 \,\,\, 8) (9 \,\,\, 11) (10 \,\,\, 12), \, (1 \,\,\, 4) (2 \,\,\, 3) (5 \,\,\, 8) (6 \,\,\, 7) (9 \,\,\, 12) (10 \,\,\, 11)
 \Big\rangle \cong C_2 \times C_2. \]
The deck transformation corresponding to the first generator is the twisted pair map $\Psi_1,$ defined just as in Equation~\eqref{eq:twistedPairMap}.
The second is a reflection-rotation symmetry $\Psi_2$ depicted in Figure~\ref{fig:4p2v_reflection}. 
To get a formula for $\Psi_2,$ it is convenient to work with the equation of the unknown plane:
\begin{equation}
\label{eq:plane}
\langle \nn , \bs{X} \rangle = d.
\end{equation}
Note that $\nn$ and $d$ depend rationally on the data.
The formula for $\Psi_2$ is given by
\begin{equation}
\begin{split}
\Psi_2 (\R) &= \R \, \left(
2\, \displaystyle\frac{\mathbf{n} \mathbf{n}^\top}{\mathbf{n}^\top \mathbf{n}}
- \I
\right)    \\
\Psi_2 (\tran) &= 
- \tran - \frac{2 d}{\nn^\top\nn} \R \nn \\
\Psi_2 (\alpha_i ) &= \alpha_i \\
\Psi_2 (\beta_i ) &= -\beta_i \\
\Psi_2 (\xx_1, \ldots, \xx_4, \yy_1, \ldots , \yy_4) &= (\xx_1, \ldots, \xx_4, \yy_1, \ldots , \yy_4).
\end{split}    
\end{equation}
To better understand the effect of $\Psi_2$ on $\tran,$ let $\bs{X}$ be any point on the scene plane and calculate
\begin{align}
 -\tran - \frac{2d}{\nn^\top\nn}\, \R  \, \nn &=  -\tran - \R \bs{X} - \R \left(2\frac{\mathbf{n}\mathbf{n}^\top}{\mathbf{n}^\top\mathbf{n}} - \I\right) \, \bs{X} \label{eq:deckX} \\
&= - \R\left( 2\frac{\mathbf{n}\mathbf{n}^\top}{\mathbf{n}^\top\mathbf{n}} - \I \right)\bigg(\Big( \I -2\frac{\mathbf{n}\mathbf{n}^\top}{\mathbf{n}^\top\mathbf{n}}\Big)\left(-\R^\top\tran - \bs{X}\right) + \bs{X}\bigg) \\
&= - \Psi_2 (\R) \, \bigg(\Big( \I -2\frac{\mathbf{n}\mathbf{n}^\top}{\mathbf{n}^\top\mathbf{n}}\Big)\left(-\R^\top\tran - \bs{X}\right) + \bs{X} \bigg) 
\end{align}
This may be understood as follows: we take $- \R^\top \tran ,$ which is the center of the second camera  expressed in the frame of the first camera (cf.\ Eq.~\ref{eq:4p2v}), then reflect this vector through the plane and transform it back to the vector representing the center of the first camera expressed in the frame of the reflected second camera (by multiplying by $-\Psi_2(\R)$).
\begin{figure}[H]
\def\svgwidth{\columnwidth}
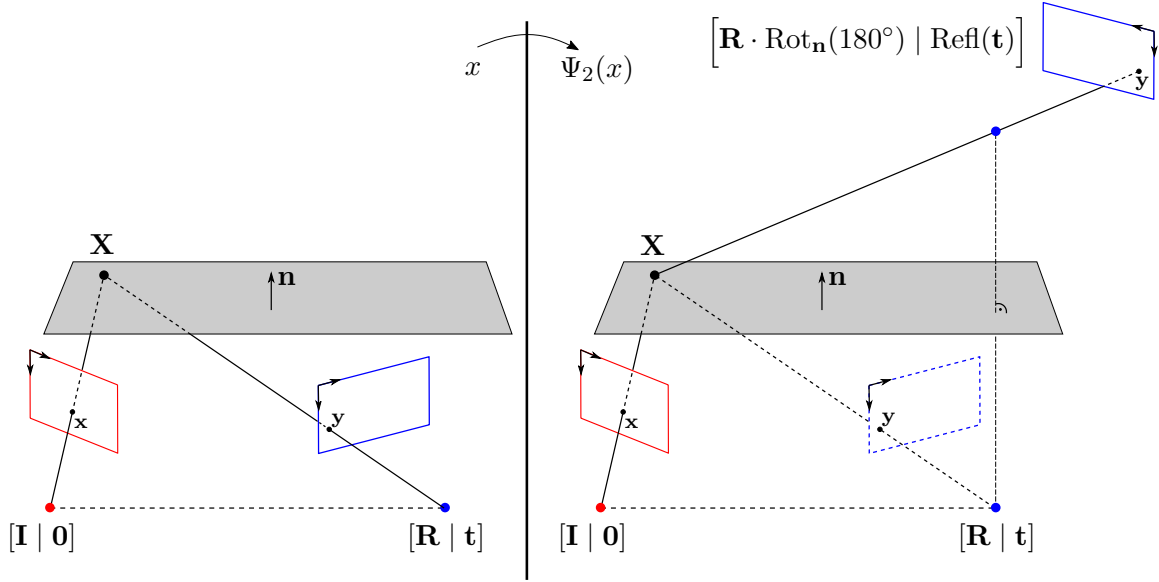
\caption{Reflection-rotation symmetry for Equations~\ref{eq:4p2v}}
\label{fig:4p2v_reflection}
\end{figure}
\begin{proposition} \label{prop:4p2v_deck}
$\Psi_1$ and $\Psi_2$ generate the deck transformation group for the planar calibrated homography problem $X\to Z$ defined by Equations~\ref{eq:4p2v}.
The corresponding permutations which centralize $\mon (X/Z)$ are as follows:
\begin{center}
\begin{tabular}{ccc}
$\Psi_1 \circ \Psi_2$ & $\leftrightarrow$ & $(1 \,\,\, 2) (3 \,\,\, 4) (5 \,\,\, 6) (7 \,\,\, 8) (9 \,\,\, 10) (11 \,\,\, 12) $\\
$\Psi_1$ & $\leftrightarrow $ & $(1 \,\,\, 3) (2 \,\,\, 4) (5 \,\,\, 7) (6 \,\,\, 8) (9 \,\,\, 11) (10 \,\,\, 12)$ \\
$\Psi_2$ & $\leftrightarrow $ & $(1 \,\,\, 4) (2 \,\,\, 3) (5 \,\,\, 8) (6 \,\,\, 7) (9 \,\,\, 12) (10 \,\,\, 11)$
\end{tabular}
\end{center}
These correspond to the maximal chains in the lattice of block systems (see Figure~\ref{fig:blocks_fields}).
\end{proposition}

\begin{proof}
We verify that $\Psi_1$ and $\Psi_2$ really are deck transformations. 
The rest is elementary or follows from Proposition~\ref{prop:deck-centralizer}.
Appealing to well-known properties of the twisted pair $\Psi_1,$ it suffices for us to check that planarity of the scene is preserved:
\[
\det \left(\begin{bmatrix}\Psi_1(\alpha_1)\xx_1 & \Psi_1(\alpha_2)\xx_2 & \Psi_1(\alpha_3)\xx_3 & \Psi_1(\alpha_4)\xx_4 \\ 1 & 1 & 1 & 1 \end{bmatrix}\right) = 0.
\]
Letting $\mathbf{m} = \frac{2}{\tran^\top\tran}\R^\top\tran ,$ we may compute this determinant as follows:
\begin{align*}
\left(\displaystyle\prod_{i=1}^4 (1+\mathbf{m}^\top\alpha_i\xx_i)\right)^{-1} \det\left(\begin{bmatrix} \alpha_1\xx_1 & \alpha_2\xx_2 & \alpha_3\xx_3 & \alpha_4\xx_4 
\\ 1+\mathbf{m}^\top\alpha_1\xx_1 & 1+\mathbf{m}^\top\alpha_2\xx_2 & 1+\mathbf{m}^\top\alpha_3\xx_3 & 1+\mathbf{m}^\top\alpha_4\xx_4 \end{bmatrix}\right) 
&=
\\ 
\left(\displaystyle\prod_{i=1}^4 (1+\mathbf{m}^\top\alpha_i\xx_i)\right)^{-1}\det\left(\begin{bmatrix} \alpha_1\xx_1 & \alpha_2\xx_2 & \alpha_3\xx_3 & \alpha_4\xx_4 \\ 1 & 1 & 1 & 1 \end{bmatrix}\right) = 0. 
\end{align*}
For $\Psi_2 ,$ it is clear that planarity of the scene is preserved and that $\Psi_2(\R) \in \SOCC (3).$
If we substitute $\Psi_2(x)$ into the point correspondence constraint in \eqref{eq:4p2v} and take $\bs{X} = \alpha_i \, \xx_i$ in Equation~\eqref{eq:deckX}, then
\[ 
-\beta_i\yy_i = \R\left(2\frac{\mathbf{n}\mathbf{n}^\top}{\mathbf{n}^\top\mathbf{n}} - \mathtt{I}\right)\alpha_i\xx_i + \left(-\tran - \R\alpha_i\xx_i - \R\left(2\frac{\mathbf{n}\mathbf{n}^\top}{\mathbf{n}^\top\mathbf{n}} - \mathtt{I}\right)\alpha_i\xx_i\right) = -\R\alpha_i\xx_i - \tran .
\]
We conclude that equations~\eqref{eq:4p2v} are invariant up to sign under application of $\Psi_2 .$
\end{proof}
Finally, we describe a decomposition of $X \to Z$
\begin{equation}\label{4p2v:factor}
X \dashrightarrow Y_1 \dashrightarrow Y_2 \to Z,    
\end{equation}
corresponding to the left-most chain in Figure~\ref{fig:blocks_fields}.
This decomposition makes use of the \deff{calibrated homography matrix} associated to $(\R , \tran)$ and the scene plane:
\begin{equation}
\label{eq:homographyMatrix}
\HH = \R + \displaystyle\frac{1}{d} \tran \nn^\top .
\end{equation}
Up to scale, any $3\times 3$ matrix has the form~\eqref{eq:homographyMatrix}.
On the other hand, any real calibrated homography matrix has an eigenvalue equal to $1$ (see eg.~\cite[Lemma 5.18]{ma2012invitation}), and thus lies on an irreducible hypersurface of degree $6$:
\[
\calH_1 = \{ \HH \in \CC^{3\times 3} \mid \det (\HH^\top \HH - \II ) = 0 \}.
\]
In our decomposition, we may take
\[
Y_1 = \{ \left(\HH , \left(
\left(\xx_1 ,\ldots , \xx_4 \right)
,
\left(\yy_1, \ldots, \yy_4
\right)
\right) \right) \in \calH_1 \times \left(\PP^2 \right)^4 \times \left( \PP^2 \right)^4 \mid \xx_i \sim \HH \yy_i , \, i =1 , \ldots 4 \}.
\]
Here we use the standard notation $\sim$ indicate that two vectors are equal up to scale.
We note that each of these correspondence constraints is equivalent to the vanishing of three homogeneous, non-independent linear equations
\[
\matrix{\xx_i}_\times\HH\yy_i = 0 .
\]
A short calculation reveals that $x\in X_z$ and $\Psi_1\circ\Psi_2(x) \in X_z$ map to the same point in $Y_1.$
We also note that $Y_1$ is irreducible, since its Zariski-open in the graph of $\ratnoname{\calH_1 \times \left (\PP^2\right)^4}{\left(\PP^2\right)^4}.$

The projection $Y_1 \to Z$ has a deck transformation given by the sign-symmetry $\HH \mapsto - \HH .$
To remove this last symmetry, we define
\begin{align}
\label{eq:Smatrix}
s &= \displaystyle\frac{1}{\HH_{1,1}^2} \nonumber \\
\SS &= \displaystyle\frac{1}{\HH_{1,1}} \, \HH 
\end{align}
and take $\ratnoname{Y_1}{Y_2} \subset \CC^9 \times \left(\PP^2\right)^4 \times \left(\PP^2\right)^4$ by mapping $\HH$ to $s$ and the 8 non-constant entries of $\SS.$
Algebraically, the ideal
\begin{equation}
\label{eq:Sideal}
\langle 
\det (\SS^\top \SS - s \II ), \, 
\matrix{\xx_1}_\times\SS \yy_1, \,
\matrix{\xx_2}_\times\SS \yy_2,\, 
\matrix{\xx_3}_\times\SS \yy_3, \,
\matrix{\xx_4}_\times\SS \yy_4
\rangle 
\end{equation}
has dimension $0$ and degree $3 = \deg (Y_2 /Z)$ for generic data $(\xx_1, \ldots , \yy_4 ) \in Z.$
The algebraic complexity as captured by the Galois group matches that of a well-known algorithm for computing $\HH,$ in which one must compute the singular values of a $3\times 3$ matrix $\lambda \HH$ recovered up to scale from the four point correspondences (see eg.~\cite[Algorithm 4.1]{HZ-2003} or~\cite[Algorithm 5.2]{ma2012invitation}).

\subsection{Three-view homography}
\label{subsec:3viewhomography}
Finally, we consider the minimal problem $\threeviewhomography $, where the task is to recover the relative orientation of three cameras from the input data of four point correspondences which lie on the incidence variety
\[
Z = \{ (\xx_1, \, \ldots , \, \xx_4, \, \yy_1 , \, \ldots , \, \yy_4, \, \zz_1 , \, \ldots , \, \zz_4) \in \left(\PP^2 \right)^{12}
\mid \xx_3 \in \agline{\xx_1}{\xx_2}, \, \yy_3 \in \agline{\yy_1}{\yy_2}, \, \zz_3 \in \agline{\zz_1}{\zz_2} \}. 
\]
Unlike in the previous section, there is no longer a twisted pair symmetry.
However, there are two symmetries analogous to the deck transformation $\Psi _2 .$
For this problem, the joint camera map defined in~\cite{PLMP} is birationally equivalent to a branched cover whose fibers are pairs of homography matrices, which are compatible in the sense that they share the same normal vector.
Thus, the solutions of interest lie on the subvariety $\mathcal{H}_2 \subset \left(\CC^{3\times 3}\right)^2$ defined to be the closed image of the map
\begin{align*}
\left(\SOCC (3) \right)^2 \times \left(\CC^3 \right)^3 &\rightarrow \left(\CC^{3\times 3}\right)^2\\
(\R_1, \R_2, \tran_1, \tran_2, \nn ) &\mapsto (R_1 + \tran_1 \nn^\top , R_2 + \tran_2 \nn^\top). 
\end{align*}
Notice that, unlike in \eqref{eq:homographyMatrix}, we have absorbed the constant $\frac{1}{d}$ into $\tran$ for each homography matrix.
We wish to compute the fibers of the branched cover $X \to Z,$ where
\begin{equation}
\label{eq:3way-homography-equations}
X = \{
\left( (\HH_1, \, \HH_2), \, (\xx_1, \, \ldots , \, \zz_4)\right) \in \mathcal{H}_2 \times Z
\mid \xx_i \sim \HH_1 \,  \yy_i \sim \HH_2 \, \zz_i , \, i=1,\ldots 4 \}.
\end{equation}
For this problem, we have $\deg (X/Z) = 64,$ and Result~\ref{res:plmp} tells us $\mon (X/Z) \cong S_2 \wr (S_2 \wr S_{16} \cap A_{32}) \cap A_{64}.$
It follows that there exists a decomposition
\[
X \dashrightarrow Y_1 \dashrightarrow Y_2 \dashrightarrow Z
\]
with $\deg (X/Y_1) =  \deg (Y_1 /Y_2) = 2$ and $\deg (Y_2/ Z) = 16.$
The deck transformations of $X \to Z$ are easily seen to be $(\HH_1 , \HH_2)\mapsto (\pm \HH_1 , \pm \HH_2).$
Thus, we may use separating invariants for this linear group action as in the previous section to write down the maps $\ratnoname{X}{Y_1}$ and $\ratnoname{Y_1}{Y_2}.$

However, our description of $X$ is unsatisfying from the point of view of constructing polynomial solvers, since we have only described $\calH_2$ parametrically.
We leave determining the ideal $\mathcal{I}_{\calH_2}$ as a challenging open problem in algebraic vision, analogous to previous works~\cite{MULTIVIEWIDEAL,TRIFOCALIDEAL}.
Our final Result~\ref{res:degree64} is a partial solution to this implicitization problem, which describes an ideal contained in $\mathcal{I}_{\calH_2}.$ 
The generators of this ideal and the linear correspondence constraints in~\eqref{eq:3way-homography-equations}
generate a $0$-dimensional ideal of  degree $64$ for generic data $z=(\xx_i, \yy_i, \zz_i).$

Drawing on the description of $\calH_1$ from the previous section, consider the map
\begin{align*}
\calH_2 &\rightarrow \left(\CC^{3\times 3}\right)^2 \\
(\HH_1,\HH_2) &\mapsto (\HH_1^\top \HH_1 - \II , \HH_2^\top \HH_2 - \I).
\end{align*}
The image of this map has the alternate parametrization
\begin{align*}
\left(\CC^{3\times 1}\right)^3 &\rightarrow \left(\CC^{3\times 3}\right)^2\\
(\dd_1, \dd_2, \nn) &\mapsto (\nn \dd_1^\top + \dd_1 \nn^\top, \nn \dd_2^\top + \dd_2 \nn^\top ).
\end{align*}
Using Macaulay2, we compute implicit equations in new matrix variables $\WW_i = \nn \dd_i^\top + \dd_i \nn^\top,$ $i=1,2.$
The resulting elimination ideal in $\CC [\WW_1, \WW_2]$ is generated by four cubics and 15 quartics.
The cubic constraints obtained are 
\begin{equation}\label{eq:cubicW}
\det (\WW_1) = \det (\WW_2) = \det (\WW_1 + \WW_2) = \det (\WW_1 - \WW_2) = 0.
\end{equation}
These cubics can be understood in terms of the 
alternate parametrization, which shows that generic $(\WW_1, \WW_2)$ in the image will span a pencil of rank-$2$ symmetric matrices.
In what in follows, it is enough for us to consider two of the 15 quartics, which have alternate expressions in terms of resultants: 
\begin{equation} \label{eq:resultants}
\begin{split}
\mathrm{Res}_{n_1} \left(\WW_1^{3,3}n_1^2 -2\WW_1^{1,3}n_1 + \WW_1^{1,1}, \WW_2^{3,3}n_1^2 -2\WW_2^{1,3}n_1 + \WW_2^{1,1}\right) = 0\\ \mathrm{Res}_{n_2}\left(\WW_1^{3,3}n_2^2 -2\WW_1^{2,3}n_2 + \WW_1^{2,2}, \WW_2^{3,3}n_2^2 -2\WW_2^{2,3}n_2 + \WW_2^{2,2}\right) = 0
\end{split}
\end{equation}
where $\nn = (n_1, n_2, n_3).$

Substituting $\WW_i = \HH_i^\top \HH_i - \II $ into~\eqref{eq:cubicW} yields four polynomials of degree $6$ vanishing on $\calH_2.$
Using Bertini~\cite{Bertini},  we computed points where these equations vanish by tracking $6^4 = 1296$ homotopy paths. 
Out of these points, $336$ lie on $\calH_2.$
The remaining $960$ paths resulted in $224$ points not on $\calH_2,$ each occurring with multiplicity $4.$
We confirmed that the degree of the variety $\calH_2$ is indeed $336$ using monodromy.

Unlike in the five-point problem, the linear equations implied by $\xx_i \sim \HH_1 \,  \yy_i \sim \HH_2 \, \zz_i $ are non-generic.
The number of solutions to these linear equations and the four degree-$6$ equations obtained from~\eqref{eq:cubicW} is $320.$
Out of these solutions, only $84$ satisfy the degree-$8$ equations obtained from~\eqref{eq:resultants}.
To obtain the degree $64$ reported in~\cite{PLMP}, it is sufficient to impose the additional constraint $\det \HH_1 \ne 0.$
In summary, we have the following result.
\begin{result}
\label{res:degree64}
For each $z=\left( \xx_1, \, \ldots , \, \zz_4\right) \in Z,$ let $I_z \subset  \CC [\HH_1, \HH_2, D] $ be the ideal in $19$ variables generated by the linear relations $\xx_i \sim \HH_1 \,  \yy_i \sim \HH_2 \, \zz_i $ for $i=1, \ldots , 4,$ 
the saturation constraint $D \det \HH_1 -1=0,$ and six equations obtained by
setting $(\WW_1, \WW_2) = (\HH_1^\top \HH_1 - \II , \HH_2^\top \HH_2 - \II ) $ in equations~\eqref{eq:cubicW} and~\eqref{eq:resultants}.
For generic $z\in Z,$ we have $\deg (I_z) = 64.$
Moreover, after substituting $s_1, s_2, \SS_1, \SS_2$ as in~\eqref{eq:Smatrix}, \eqref{eq:Sideal}, we obtain an ideal of degree $16$ for generic data.
\end{result}

\section{Conclusion and outlook}
\label{sec:outlook}

Galois/monodromy groups reveal the intrinsic algebraic structure of problems in enumerative geometry.
We have shown how this structure can be revealed on both new and old examples from computer vision.
We believe that numerical algebraic geometry and computational group theory are valuable additions to the arsenal of techniques used by researchers interested in building minimal solvers.
Although these methods allow us to identify decomposable branched covers, the problem of automatically computing a decomposition as in Equation~\eqref{eq:decomposition} seems hard in general.
We have shown how understanding the underlying geometry allows us to solve it in cases of interest.
Finally, we note that the techniques in this paper might also be applied to non-minimal problems, where the branched covers $\ratnoname{X}{Z}$ of interest are such that $Z$ is a low-dimensional subvariety of some ambient space of data.
In principal, the Galois/monodromy group may be computed by projecting $Z$ onto an affine space of the same dimension and applying Part 2) of Proposition~\ref{prop:factor}.
This, however, lies beyond the scope of our work.

\section*{Acknowledgements}
We are grateful to ICERM (NSF DMS-1439786 and the Simons Foundation grant 507536) for their support and for hosting three events where significant progress on this project was made: the Fall 2018 semester program on Nonlinear Algebra, the Winter 2019 Algebraic Vision research cluster, and the Fall 2020 virtual workshop on Monodromy and Galois groups in enumerative geometry and applications. V.~Korotynskiy and T.~Pajdla were supported by the EU Structural and Investment Funds, Operational Programe Research, Development and Education under the project IMPACT (reg.\ no.\ CZ$.02.1.01/0.0/0.0/15\_003/0000468$), the EU H2020 ARtwin No.~856994, and EU H2020 SPRING No.~871245 Projects. T.~Duff received partial support from NSF DMS \#1719968 and \#2001267. M.~Regan was supported in part by Schmitt Leadership Fellowship in Science and Engineering and NSF grant CCF-1812746.  We are particularly grateful to Jon Hauenstein, Anton Leykin, and Frank Sottile for helpful suggestions. 

\bibliographystyle{amsplain}
\bibliography{refs,local}

\end{document}